\documentclass[a4paper]{amsart}
\usepackage{color}
\usepackage{amssymb}
\usepackage{comment}
\usepackage{graphicx}
\usepackage{pstricks}
\usepackage{accents}
\usepackage{xspace}

\numberwithin{equation}{section}

\theoremstyle{plain}
\newtheorem{thm}{Theorem}[section]
\newtheorem{cor}[thm]{Corollary}
\newtheorem{prop}[thm]{Proposition}
\newtheorem{lem}[thm]{Lemma}

\theoremstyle{definition}
\newtheorem{rem}[thm]{Remark}
\newtheorem{definition}[thm]{Definition}

\newcommand{\R}{\mathbb{R}}
\newcommand{\diam}{\mathrm{diam}}
\newcommand{\dist}{\mathrm{dist}}
\newcommand{\tri}{\mathcal{T}}              
\newcommand{\faces}{\mathcal{E}_\mathrm{i}} 
\newcommand{\Afaces}{\mathcal{E}}           
\newcommand{\Bfaces}{\mathcal{E}_\partial}  
\newcommand{\el}{K}                         
\newcommand{\fa}{E}                         
\newcommand{\pair}[1]{\omega_\tri(#1)}      
\newcommand{\Mpair}[1]{\omega_\star(#1)}    
\newcommand{\SetPair}{\mathcal{W}_\tri}     
\newcommand{\Sets}{covering}                
\newcommand{\Patch}[1]{\widetilde{\omega}_\tri(#1)}     
\newcommand{\Skeleton}[1]{\gamma_{\tri}(#1)}            
\newcommand{\skeleton}[1]{\bar{\gamma}_{\tri}(#1)}      
\newcommand{\ShapePar}{\sigma}                          
\newcommand{\FEspace}{S}                        
\newcommand{\LFEspace}[1]{\FEspace|_{#1}}       

\newcommand{\lNnn}{\lVert {\hskip -0.1em} \lvert}
\newcommand{\rNnn}{\rVert {\hskip -0.1em} \rvert}
\newcommand{\tnorm}[1]{\lNnn #1 \rNnn}      
\newcommand{\norm}[1]{\left\|#1\right\|}    
\newcommand{\grd}{\nabla}
\newcommand{\veps}{\varepsilon}

\newcommand{\RitzP}[1]{\mathcal{R}^{\veps}_{#1}}    
\newcommand{\LtwoP}[1]{\mathcal{P}_#1}      
\newcommand{\BestApp}[1]{\mathcal{Q}_{#1}}  
\newcommand{\interp}{\Pi}                   

\newcommand{\simp}{T}
\newcommand{\elAv}{K_z}

\newcommand{\Cnorm}{C_N}                    
\newcommand{\cnorm}{c_N}
\newcommand{\constant}{M}
\newcommand{\Cloc}{C_{\mathrm{loc}}}

\newcommand{\hE}[1]{h_{\pair{#1}}}

\newcommand{\err}[1]{e(#1)}                 
\newcommand{\Err}{E}                        
\newcommand{\tree}{B}                       
\newcommand{\mtree}{B^{*}}                  
\newcommand{\leaves}[1]{\mathcal{L}(#1)}    
\newcommand{\intNodes}[1]{N(#1)}            

\newcommand{\inters}[1]{\widetilde{#1}}
\newcommand{\ed}{S}

\begin{document}
\title[Robust Error Localization in the Reaction-Diffusion Norm]%
{Robust Localization of the Best Error with Finite Elements in the 
Reaction-Diffusion Norm}
\author[F.~Tantardini]{Francesca Tantardini}
\address[Francesca Tantardini and Andreas Veeser]
  {Dipartimento di Matematica\\
  Universit\`a degli Studi di Milano\\
  Via Saldini 50\\
  20133 Milano\\
  Italy}
\email[Francesca Tantardini]{francesca.tantardini1@unimi.it}
\author[A.~Veeser]{Andreas Veeser}
\email[Andreas Veeser]{andreas.veeser@unimi.it}
\urladdr[Andreas Veeser]{users.mat.unimi.it/users/veeser/}
\author[R.~Verf\"urth]{R\"udiger Verf\"urth}
\address[R\"udiger Verf\"urth]
  {Fakult\"at f\"ur Mathematik\\
  Ruhr-Universit\"at Bochum\\
  Uni\-ver\-si\-t\"ats\-stra{\ss}e 150\\
  44801 Bochum\\
  Germany}
\email[R\"udiger Verf\"urth]{ruediger.verfuerth@ruhr-uni-bochum.de}
\urladdr[R\"udiger Verf\"urth]{www.ruhr-uni-bochum.de/num1/}
\begin{abstract}
We consider the approximation in the reaction-diffusion norm with continuous 
finite elements and prove that the best error is equivalent to a 
sum of the local best errors on pairs of elements.  The equivalence 
constants do not depend on the ratio of diffusion to reaction.  As 
application, we derive local error functionals that ensure robust performance 
of adaptive tree approximation in the reaction-diffusion norm. 
\end{abstract}


\maketitle

\section{Introduction}
%
%
Finite element methods are well-established for the numerical solution of 
elliptic and parabolic problems.  An important aspect in their mathematical 
understanding and foundation are the approximation properties of finite 
elements spaces.  In view of adaptive mesh refinement, the local features of 
the latter under minimal regularity assumptions are of interest.

The most basic finite element approach to the homogeneous Dirichlet problem 
for Poisson's equation leads to the following approximation problem:  
Approximate a function $u\in H^1_0(\Omega)$ in the $H^1$-seminorm with 
functions from a space $\FEspace$ consisting of continuous piecewise 
polynomials of degree $\leq\ell$ associated with a given simplicial mesh 
$\tri$.  
In this context one of the authors \cite{Veeser:13} proved that 
\begin{equation}
\label{E:H1dec}
 \inf_{v\in \FEspace}\norm{\grd(u-v)}_\Omega
 \approx
 \left(
  \sum_{\el\in\tri}\inf_{P\in\mathbb{P}_\ell(\el)}\norm{\grd(u-P)}^2_\el
 \right)^{1/2},
\end{equation}
i.e.\ the global best error is equivalent to the $\ell_2$-norm of the local 
best errors on elements.  Notice that the right-hand side does not involve any 
coupling between elements and that no additional regularity of $u$ is invoked.  
If $u$ disposes of additional piecewise regularity, this result and the 
Bramble-Hilbert Lemma readily imply error bounds.  Moreover, it shows that 
adaptive tree approximation \cite{Binev.DeVore:04} by P.~Binev and 
R.~DeVore with the local best errors as error functionals yields near best 
meshes for the global best error on the left-hand side.

In view of problems with extreme parameters, it is important that approximation 
properties are robust.  An important and basic example for such a problem is 
given by reaction-dominated diffusion, whose stationary variant is also of 
interest in the discretization of the heat equation.  In this context the 
$H^1$-seminorm in \eqref{E:H1dec} is replaced by the so-called 
reaction-diffusion norm
\begin{equation}\label{rd-norm}
 \tnorm{\cdot}^2
 :=
 \norm{\cdot}^2+\veps\norm{\grd\cdot}^2,
 \qquad\text{with }
 \veps>0,
\end{equation}
and one is interested in a variant of \eqref{E:H1dec} where the hidden 
constants are independent of the parameter $\veps$.  The exact counterpart of 
\eqref{E:H1dec} for the reaction-diffusion norm cannot be robust; this arises 
from the fact that, for $\veps=0$, a discontinuous piecewise constant function 
yields $0$ for the sum of the local best errors, but not in general for the 
global best error.  We find the following robust replacements:
\begin{equation}
\label{replacements}
\begin{aligned}
 \inf_{v\in \FEspace}\tnorm{u-v}_\Omega
 &\approx
 \left(\sum_{\fa\in\faces}
  \inf_{P\in\LFEspace{\pair{\fa}}} \tnorm{u-P}^2_{\pair{\fa}}\right)^{1/2}
\\
 &\lesssim \left(
  \sum_{\el\in\tri}\inf_{P\in\mathbb{P}_\ell(\el)}
   \left(
    \tnorm{u-P}^2_\el
    + \frac{|\el|}{|\partial\el|} \norm{u-P}^2_{\partial\el}
   \right)
 \right)^{1/2},
\end{aligned}
\end{equation}
where the infima in the first sum are localized to continuous piecewise 
polynomials on pairs of elements $\pair{\fa}$ sharing an internal face 
$\fa\in\faces$. Notice that the second sum is ready to apply the 
Bramble-Hilbert lemma, while the first one is not.  The second sum however does 
not provide a robust lower bound of the global best error on the left-hand 
side. The reason lies in the fact that, for $\veps=0$, it requires an 
additional $1/2$-derivative.  Accordingly, adaptive tree approximation with the 
local contributions of the second sum cannot provide near best meshes for the 
reaction-diffusion in a robust manner.  The local contributions of the first 
sum are also not suitable for adaptive tree approximation, but for another 
reason: they do not allow to define local error functionals depending solely 
on the target function and a given element.   Adopting however the idea of 
minimal rings in \cite{Binev.DeVore:04} to pairs, we provide a modification of 
the first sum that is suitable for tree approximation.

The article is organized as follows. In \S\ref{S:counterexample} we show that 
the hidden constant of the nontrivial inequality of \eqref{E:H1dec} for the 
reaction-diffusion norm blows up for $\veps\searrow0$.  In \S\ref{S:aux_res} 
we fix notations, while in \S\ref{S:loc_int} we show that localization 
results like the first part of \eqref{replacements} follow from a suitable 
property of a quasi-interpolation operator.  This is exploited, in 
\S\ref{S:pair} and \S\ref{S:i4ta} respectively, to prove 
the first part of \eqref{replacements} and its counterpart for minimal 
pairs.  Section \S\ref{alternative} analyzes the non-robustness of
\S\ref{S:counterexample} more precisely, thereby deriving an alternative 
way to compute the local best errors in the first part of 
\eqref{replacements} and showing its second part. Finally, we extend in 
\S\ref{S:H10} our results to conforming approximation of functions with 
vanishing boundary values.

\section{Decoupling of elements is not robust}
\label{S:counterexample}
%
%
The purpose of this section is to show that, for the reaction-diffusion norm 
$\tnorm{\cdot}$, the `$\lesssim$'-part in \eqref{E:H1dec} cannot hold with a 
constant independent of $\veps$. The counterexample provides functions 
$u_\veps\in H^1(\Omega)$ converging to a discontinuous function $u_0\notin 
H^1(\Omega)$ such that the global best error is bounded from below 
independently of $\veps$, while the local best errors decrease with 
$\sqrt[4]\veps$.

We consider the domain $\Omega=(-2,2)\times(-1,1)\subset\R^2$, with 
the subdomains $\Omega_+=(0,2)\times(-1,1)$ and $\Omega_-=(-2,0)\times(-1,1)$. 
Let $\tri$ be any conforming simplicial triangulation of $\Omega$, that is 
subordinate to the decomposition $\bar{\Omega} = \bar{\Omega}_+ \cup 
\bar{\Omega}_-$, and let $\FEspace$ be the space of continuous functions that 
are piecewise polynomial of degree at most $\ell$ with respect to $\tri$.  We 
denote by $\LtwoP{\tri}$ the $L^2$-projection onto $\FEspace$, by 
$\RitzP{\tri}$ the Ritz projection onto $\FEspace$ with respect to the 
reaction-diffusion norm and by $\RitzP{\el}$ the local counterpart of 
$\RitzP{\tri}$.  The functions $u_\veps\in H^1(\Omega)$ and $u_0\in 
L^2(\Omega)$ 
are given by
\begin{align*}
 u_\veps(x)
 =
 \begin{cases}
  -1 &\text{for }x_1<-\sqrt\veps,
 \\            
  \frac{x_1}{\sqrt\veps} &\text{for }-\sqrt\veps<x_1<\sqrt\veps,
 \\
  1 &\text{for }\sqrt\veps<x_1,
 \end{cases}
\qquad
 u_0(x)
 =
 \begin{cases}
  -1 &\text{for }x_1<0,
 \\
  1 &\text{for }0<x_1.
 \end{cases}
\end{align*}
On one hand, for the local best errors, $u_0\in\mathbb{P}_\ell(\el)$ for every 
$\el\in\tri$ implies
\begin{align*}
 \sum_{\el\in\tri}\tnorm{u_\veps-\RitzP{\el}u_\veps}^2_\el
 &\leq
 \sum_{\el\in\tri}\tnorm{u_\veps-u_0}_\el^2
\\
 &=
 \int_{-\sqrt\veps}^{\sqrt\veps}\int_{-1}^1\veps
  \left|\frac{\partial u_\veps}{\partial x_1}\right|^2
  + \int_{-\sqrt\veps}^{\sqrt\veps}\int_{-1}^1 
   \left|u_\veps-u_0\right|^2=\frac{16}{3}\sqrt\veps.
\end{align*}
On the other hand, for the global best error, there holds
\begin{align*}
 \tnorm{u_\veps-\RitzP{\tri}u_\veps}_{\Omega}
 &\geq
 \norm{u_\veps-\RitzP{\tri} u_\veps}_{L^2(\Omega)}
\\
 &\geq\norm{u_\veps-\LtwoP{\tri}u_\veps}_{L^2(\Omega)}
 \rightarrow
 \norm{u_0-\LtwoP {\tri}u_0}_{L^2(\Omega)}>0
\qquad \text{for }\veps\rightarrow 0,
\end{align*}
since $u_\veps$ converges to $u_0$ in the $L^2$-norm, the $L^2$-projection onto 
$\FEspace$ is continuous and $u_0\notin \FEspace$.  Consequently, the constant 
in the nontrivial inequality of \eqref{E:H1dec} for the reaction-diffusion norm
has to grow at least with $\veps^{-1/4}$.

This simple example reflects a more general situation. Consider in fact any 
couple of adjacent elements $\el$ and $\widetilde{\el}$. Taking 
$u_\veps=\min\{1,\veps^{-1/2}\dist(\mathbb{R}^d\setminus \el)\}$ and 
$u_0=\chi_\el$ and reasoning as above shows that the constant blows up as 
$\veps\searrow0$.   These observations suggest to modify \eqref{E:H1dec} for the
reaction-diffusion norm by invoking local best errors that incorporate the
continuity constraint  through a face.

\section{Meshes and basis functions}
\label{S:aux_res}
%
%
%
%

We denote by $\tri$ a  conforming simplicial mesh of a polyhedral domain 
$\Omega\subset\mathbb{R}^d$, by $\Afaces$ the set of its faces, and by $\faces$ 
the subset of $\Afaces$ of those faces which are not contained in the boundary 
$\partial\Omega$.  If $\el\in\tri$ is an element and $\fa\in\Afaces$ is a face, 
we write $|\el|$ and $|\fa|$ for its $d$-dimensional Lebesgue and 
$(d-1)$-dimensional Hausdorff measure, respectively.
For every face $\fa\in \Afaces$, the set
\[
 \pair{\fa} := \bigcup\{\el\in\tri: \partial\el\supseteq\fa \}
 \]
is the union of elements sharing the face $\fa$.  It consists of two elements 
if $\fa\in\faces$ and of one element otherwise.
We stress that $\fa$ belongs to various meshes and that $\pair{\fa}$ 
actually depends on $\tri$ too.

A collection $\mathcal{W}$ of subdomains of $\Omega$ is a 
\emph{$\beta$-finite \Sets} of $\tri$ if for every $\el\in\tri$
\begin{itemize}
 \item there exists $\omega\in\mathcal{W}$ with $\omega\supseteq\el$ and 
 \item there holds $\sum_{\omega\in\mathcal{W}} \chi_\omega\leq\beta$ on 
$\accentset{\circ}{\el}$,
\end{itemize}
where $\chi_\omega$ stands for the characteristic function of $\omega$.  The 
collections $\{\el\}_{\el\in\tri}$ and 
$\SetPair:=\{\pair{\fa}\}_{\fa\in\faces}$ in \eqref{E:H1dec} and 
\eqref{replacements} are 1-finite and $(d+1)$-coverings, respectively.  
Notice that $\beta$ arises in \eqref{E:H1dec} and \eqref{replacements} as 
multiplicative constant in the straight-forward inequality.  Another 
$\beta$-finite covering appears in \S\ref{S:i4ta}.

The space
\[
 \FEspace
 :=
 S^{\ell,0}(\tri)
 =
 \{v\in C^{0}(\Omega): \,
  v\in\mathbb{P}_\ell(\el), \,\forall\, \el\in\mathcal{\tri}\}
\]
consists of all continuous functions that are piecewise polynomial over $\tri$. 
Given a set $\omega\subset\Omega$, we indicate its restriction with
\[
 \LFEspace{\omega}
 :=
 \{v\in C^0(\omega):
 \exists \tilde{v}\in \FEspace, \ \tilde{v}|_{\omega}=v\}.
\]
In particular we have
$\LFEspace{\pair{\fa}}=\{v\in C^{0}(\pair{\fa}): \,v\in\mathbb{P}_\ell(\el), 
\,\forall \el\in\tri,\ \el\subseteq\pair{\fa}\}$.
Furthermore, we denote by $\mathcal{N}$ the set of nodes of $\FEspace$. A 
subscript $\el$, $\Omega$, etc.\ to $\mathcal{N}$ indicates that only the nodes 
contained in the index-set are considered. We denote by 
$\{\phi_z\}_{z\in\mathcal{N}}$ the nodal basis, that is, for every 
$z\in\mathcal{N}$,
\[
 \phi_z\in \FEspace
 \qquad\text{and}\qquad
 \forall y\in\mathcal{N} \quad \phi_{z}(y)=\delta_{yz}.
\]
Given an element $\el\in\tri$, the $L^2(\el)$-dual basis functions 
$\{\psi_z^\el\}_{z\in\mathcal{N}_\el}$ are such that
\[
 \int_\el\psi_z^\el\phi_y=\delta_{zy}.
\]
for every $y$, $z\in\mathcal{N}_\el$.  We thus have, for every 
$p\in\mathbb{P}_\ell(\el)$ and for every $z\in\mathcal{N}_\el$,
\begin{equation}
\label{psi_prop} p(z)=\int_\el p\psi_z^\el.
\end{equation}
We also recall some basic scaling properties of different norms of $\phi_z$ 
and $\psi_z^\el$. We denote by $\Hat{\el}:=\text{conv}\{0,e_1,\dots,e_d\}$ the 
reference $d$-simplex, by $\Hat{h}:=\diam(\Hat{\el})$ the diameter of 
$\Hat{\el}$, by $\{\Hat{\phi}_{\Hat{z}}\}$, $\{\Hat{\psi}_{\Hat{z}}\}$ 
respectively the basis and dual basis functions on $\Hat{\el}$, and by 
$\norm{\cdot}_\omega$ the $L^2$-norm on the set $\omega$. For every element 
$\el\in\tri$, there exists an affine transformation $F:\mathbb{R}^d\rightarrow 
\mathbb{R}^d$ with $F(\Hat{\el})=\el$, and $F(\Hat{z})=z$.
There may be different choices for $\Hat{z}$, which nevertheless lead to the 
same value of $\|\Hat{\phi}_{\Hat{z}}\|_{\Hat{\el}}$. However, 
$\|\nabla\Hat{\phi}_{\Hat{z}}\|_{\Hat{\el}}$ depends on the chosen node. For 
this reason, we take a $\Hat{z}$ with minimal sum of the coordinates, so that 
$\|\nabla\Hat{\phi}_{\Hat{z}}\|_{\Hat{\el}}$ is unique.
Since $\Hat{\psi}_{\Hat{z}}=(\det B)\psi_z^\el\circ F$, where $B$ is the 
non-singular matrix associated to $F$, the transformation rule and the proof 
of \cite[Theorems 15.1 and 15.2]{Ciarlet:91} imply
\begin{align}
\label{phi_scaling}
 \norm{\phi_z}_\el
 &=
 \frac{|\el|^{1/2}}{|\Hat{\el}|^{1/2}}\|\Hat{\phi}_{\Hat{z}}\|_{\Hat{\el}},
\\
 \label{psi_scaling}
 \norm{\psi_z^\el}_{\el}
 &=
 \frac{|\Hat{\el}|^{1/2}}{|\el|^{1/2} }\|\Hat{\psi}_{\Hat{z}}\|_{\Hat{\el}},
\\
 \label{phi_scaling_grad}
 \norm{\nabla\phi_z}_{\el}
 &\leq
 \frac{\Hat{h}|\el|^{1/2}}{\rho_\el|\Hat{\el}|^{1/2}} 
  \|\nabla\Hat{\phi}_{\Hat{z}} \|_{\Hat{\el}},
\end{align}
where $\rho_\el$ denotes the maximum diameter of a ball inscribed in $\el$.

\section{Localization and interpolation}
\label{S:loc_int}
%
%
In this section we reduce the problem of localizing the global best error to 
the problem of defining a global quasi-interpolation operator that is 
locally near best.  Roughly speaking, the latter means that the difference 
between the interpolant and a local best approximation is bounded, up to a 
constant, by a finite sum of local best errors.  The results of this section 
are used in \S\ref{S:pair} and \S\ref{S:i4ta}.

Let $\tnorm{\cdot}_{\Omega}$ be a norm on $H^1(\Omega)$ such that its square is
set-additive. Then there holds, e.g.\ $\tnorm{\cdot}^2_{\Omega} = 
\sum_{\el\in\mathcal{\tri}}\tnorm{\cdot}^2_\el$.  Moreover let $\mathcal{W}$ be
a $\beta$-finite \Sets. For every subdomain $\omega\in\mathcal{W}$, let
$\BestApp{\omega}:H^1(\omega)\rightarrow \LFEspace{\omega}$ be a local
operator which maps a function to a corresponding best approximation in
$\LFEspace{\omega}$ with respect to $\tnorm{\cdot}_\omega$.  We thus have 
\begin{equation}
\label{hpQ}
 \tnorm{u-\BestApp{\omega}u}_{\omega}
 =
 \inf_{v\in \LFEspace{\omega}}\tnorm{u-v}_{\omega}
\end{equation}
for all $\omega\in\mathcal{W}$ and every $u\in H^1(\omega)$.
\begin{definition}[Local near best interpolation]
\label{interp_prop}
An interpolation opertor $\interp$ is locally near best with respect to
$\tnorm{\cdot}_\omega$, $\omega\in\mathcal{W}$, if there are constants 
$\Cloc\geq0$ and $\alpha_1,\alpha_2\geq 1$ with the following property:
for every element $\el\in\tri$, there exists a set $A_\el\subset\mathcal{W}$ and
a subdomain $\omega\in A_\el$ such that $\omega\supseteq\el$ and
\begin{equation}
\label{hp1}
 \sum_{z\in \mathcal{N}_{\el}}
  |\BestApp{\omega}u(z)-\interp u(z)| \, \tnorm{\phi_z}_{\el}
 \leq
 \Cloc \sum_{\omega'\in A_\el}\tnorm{u-\BestApp{\omega'}u}_{\omega'},
\end{equation}
where the set $A_\el$ consists of ``neighbouring'' subdomains, subject to the
following two conditions:
\begin{itemize}
\item there holds $\# A_\el\leq\alpha_1$ for every $\el\in\tri$,
\item for every $\omega\in\mathcal{W}$, there are at most $\alpha_2$ elements
$\el'\in\tri$ with $\omega\in A_{\el'}$.
\end{itemize}
\end{definition}
The conditions on the sets $A_\el$, $\el\in\tri$, help that local near best
interpolation entails global near best interpolation.   This reveals the proof 
of the 
following proposition.
\begin{prop}[Localization by interpolation]
\label{P:key_estimate}
If there exists an interpolation operator into $\FEspace$ that is locally near 
best with respect to $\tnorm{\cdot}_\omega$, $\omega\in\mathcal{W}$, then 
\begin{equation}
\label{th1}
 \inf_{v\in \FEspace}\tnorm{u-v}_{\Omega}
 \leq 
 C\left(
  \sum_{\omega\in\mathcal{W}}
   \inf_{v\in \LFEspace{\omega}}\tnorm{u-v}_{\omega}^2
   \right)^{1/2}
 \leq
 C\beta \inf_{v\in \FEspace}\tnorm{u-v}_{\Omega}
\end{equation}
with $C=\sqrt{\alpha_1\alpha_2}(1+\Cloc)$ for every $u\in H^1(\Omega)$.
\end{prop}
\begin{proof}
Let $\interp:H^1(\Omega)\to\FEspace$ be an interpolation operator that is 
locally near best with respect to $\tnorm{\cdot}_\omega$, 
$\omega\in\mathcal{W}$.  Bounding the infimum on the left-hand side of 
\eqref{th1} by $\tnorm{u-\interp u}_{\Omega}$ and writing the norm as 
a sum over elements results in
\[
 \inf_{v\in \FEspace}\tnorm{u-v}_{\Omega}
 \leq
 \tnorm{u-\interp u}_{\Omega}
 =
 \left(
  \sum_{\el\in\mathcal\tri}\tnorm{u-\interp u}_{\el}^2
 \right)^{1/2}.
\]
Fix an element $\el\in\tri$ and choose a subdomain $\omega\supseteq\el$ as in 
Definition \ref{interp_prop}.  The triangle inequality then yields
\[
 \tnorm{u-\interp u}_{\el}
 \leq
 \tnorm{u-\BestApp{\omega} u}_{\el}
 +
 \tnorm{\interp u-\BestApp{\omega} u}_{\el}.
\]
Since both $\interp u|_{\el}$ and $\BestApp{\omega} u|_{\el}$ are in 
$\mathbb{P}_\ell(\el)$, we can represent them in terms of the local nodal basis 
$\{\phi_z\}_{z\in\mathcal{N}_{\el}}$ and obtain
\[
 \tnorm{\interp u-\BestApp{\omega} u}_{\el}
 \leq
 \sum_{z\in \mathcal{N}_{\el}}
  |\BestApp{\omega} u(z) - \interp u(z)| \, \tnorm{\phi_z}_{\el}.
\]
Using \eqref{hp1} and inserting back up to the first inequality, we get
\[
 \inf_{v\in \FEspace}\tnorm{u-v}_{\Omega}
 \leq
 \left(
  \sum_{\el\in\tri}\left[
   (1+\Cloc) \sum_{\omega'\in A_\el} 
    \tnorm{u-\BestApp{\omega'}u}_{\omega'}
   \right]^2
 \right)^{1/2}.
\]
As $\# A_\el\leq\alpha_1$ we can use $\left(\sum_{i=1}^n a_i\right)^2\leq 
n\sum_{i=1}^na_i^2$ with $n\leq\alpha_1$.  Moreover every 
$\omega\in\mathcal{W}$ belongs to at most $\alpha_2$ of the sets $A_{\el}$ 
and so, upon rearranging terms, we arrive at
\[
 \inf_{v\in \FEspace}\tnorm{u-v}_{\Omega}
 \leq
 \sqrt{\alpha_1\alpha_2}(1+\Cloc)
 \left(
  \sum_{\omega\in\mathcal{W}}
  \tnorm{u-\BestApp{\omega}u}_{\omega}^2
 \right)^{1/2}.
\]
In view of \eqref{hpQ}, this proves the first inequality in \eqref{th1}.  To 
verify the second one, take $v\in\FEspace$, observe 
$\tnorm{u-\BestApp{\omega}u}_{\omega}\leq\tnorm{u-v}_{\omega}$ for any 
subdomain $\omega\in\mathcal{W}$ and recall that every element $\el\in\tri$ 
appears in at most $\beta$ subdomains $\omega\in\mathcal{W}$.\qedhere
\end{proof}

Our task is now reduced to find an operator $\interp$ that is locally near 
best.  Since point values are in general not defined for the 
norms of our interest, the definition of $\interp$ below in \S\ref{S:L2-norm} 
typically entails that $\interp u|_{\el}$ on some element $\el\in\tri$ depends 
also on $u|_{\widetilde{\el}}$ on certain other elements $\widetilde{\el}$.  To 
deal with this dislocation, we invoke suitable paths of overlapping subdomains. 
The following proposition applies also to the covering in \S\ref{S:i4ta}, whose 
subdomains are in general not unions of elements of the mesh $\tri$.

\begin{prop}[Dislocation control]
\label{P:path}
Let $\mathcal{W}$ be a $\beta$-finite {\Sets} of $\tri$ and $\el, 
\widetilde{\el}\in\tri$ two elements sharing a node $z\in\mathcal{N}$. 
If there exist a finite sequence $\{\omega_j\}_{j=1}^n\subset\mathcal{W}$ 
and $\nu\in(0,1]$ such that
\begin{itemize}
\item $\omega_1\supseteq\el$ and $\omega_n\supseteq\widetilde{\el}$,
\item any intersection $\omega_j\cap\omega_{j+1}$ is a simplex $\simp_j$ 
containing $z$ and there is an element $\el_j\in \tri$ which again contains 
$z$ and satisfies $|\simp_j|\geq\nu|\el_j|$,
\end{itemize}
then
\begin{align*}
 \left|
  \BestApp{\omega_1}u(z)-\int_{\widetilde{\el}}u\psi_z^{\widetilde{\el}}
 \right|
 \leq
 2\nu^{-1/2}\|\Hat{\psi}_{\Hat{z}}\|_{\Hat{\el}}
 \sum_{j=1}^{n}
  \frac{|\Hat{\el} |^{1/2}}{|\el_j|^{1/2}}
  \norm{\BestApp{\omega_j}u-u}_{L^2(\omega_j)}.
\end{align*}
with $\el_n=\widetilde{\el}$.
\end{prop}
Comparing with Proposition \ref{P:key_estimate}, we notice that the bound 
involves not best errors but $L^2$-errors of best approximations.
\begin{proof}
For every $j=2,\ldots,n$, we add and subtract $\BestApp{\omega_j}u(z)$, which 
is well-defined thanks to $z\in\omega_j$, and use the triangle inequality to get
\[
 \left|
  \BestApp{\omega_1}u(z)-\int_{\widetilde{\el}}u\psi_z^{\widetilde{\el}}
 \right|
 \leq
 \left|
  \BestApp{\omega_n}u(z)-\int_{\widetilde{\el}}u\psi_z^{\widetilde{\el}}
\right|+\sum_{j=1}^{n-1}\left|\BestApp{\omega_j}u(z)-\BestApp{\omega_{j+1}}
u(z)\right|.
\]
We bound the terms on the right-hand side separately.  Exploiting property 
\eqref{psi_prop} and the Cauchy-Schwarz inequality, we obtain
\[
 \left|
  \BestApp{\omega_n}u(z)-\int_{\widetilde{\el}}u\psi_z^{\widetilde{\el}}
 \right|
 =
 \left|
  \int_{\widetilde{\el}}(\BestApp{\omega_n}u-u)\psi_z^{\widetilde{\el}} 
 \right|
 \leq 
 \norm{\BestApp{\omega_n}u-u}_{L^2(\widetilde{\el})}
  \norm{\psi_z^{\widetilde{\el}}}_{ L^2(\widetilde{\el})},
\]
and similarly, for every $j=1,\ldots,n-1$,
\begin{align*}
 &\left|
  \BestApp{\omega_j}u(z)-\BestApp{\omega_{j+1}}u(z)
 \right|
 =
 \left|
  \int_{\simp_j}
   (\BestApp{\omega_j}u-\BestApp{\omega_{j+1}}u)\psi_z^{\simp_j} 
 \right|
\\&\qquad
 \leq
 \norm{\BestApp{\omega_j}u-\BestApp{\omega_{j+1}}u}_{L^2(\simp_j)}
 \norm{\psi_z^{\simp_j}}_{L^2(\simp_j)}
\\&\qquad
 \leq 
 \left(
  \norm{\BestApp{\omega_j}u-u}_{L^2(\simp_j)}
  +\norm{\BestApp{\omega_{j+1}}u-u}_{L^2(\simp_{j})}
 \right)
 \norm{\psi_z^{\simp_j}}_{L^2(\simp_j)}.
\end{align*}
When summing the last inequality over $j$, the $L^2$-norm of 
$\BestApp{\omega_j}u-u$ appears on both $\simp_j$ and $\simp_{j-1}$.  We bound 
both contributions by the $L^2$-norm on $\omega_j$ and combine this with the 
scaling property \eqref{psi_scaling} of $\psi_z$ and $|\simp_j|\geq\nu|\el_j|$. 
Finally, for simplification, we incorporate the term with $\widetilde\el$ into 
the 
sum and obtain the claimed inequality.
\end{proof}
For the covering $\SetPair=\{\pair{\fa}\}_{\fa\in\faces}$, the existence of 
the path in Proposition \ref{P:path} follows from the following property of 
the mesh $\tri$; see \cite{Veeser:13} for a discussion of its implications on 
the regularity of the boundary $\partial\Omega$.
\begin{definition}[Face-connectedness]
\label{tri_prop}
A simplicial mesh $\tri$ is face-connected, if for every element pair $\el$, 
$\widetilde{\el}\in\tri$ sharing a node $z\in\mathcal{N}$, there exists a 
pairwise disjoint finite sequence $\{\el_j\}_{j=1}^n\subset\tri$ such that
\begin{itemize}
 \item $\el_1=\el$ and $\el_n=\widetilde{\el}$,
 \item each intersection $\el_j\cap\el_{j+1}\in\faces$ is an interelement face 
containing $z$.
\end{itemize}
The length $n$ of the path is bounded in terms of
\begin{equation}
\label{Barn}
 \Bar{n}
 :=
 \max_{z\in\mathcal{N}} \#\{\el\in\tri: \el\ni z\}.
\end{equation}
\end{definition}
Essentially, Propositions \ref{interp_prop} and \ref{P:path} cover both 
\eqref{E:H1dec} and \eqref{replacements}; in fact, \cite{Veeser:13} uses a 
variant of Proposition \ref{P:path}, where the intersection of subdomains are 
faces.  The following definition specifies the property of the covering 
$\mathcal{W}$ that is crucial for the robustness in the first part of 
\eqref{replacements}.
\begin{definition}[Internal face covering]
\label{e:ric_E}
A $\beta$-finite covering $\mathcal{W}$ covers interelement faces internally 
if for every interelement face $\fa\in\faces$, there exists 
$\omega\in\mathcal{W}$ such that its interior almost contains $\fa$, i.e.\ 
$\fa\subset\omega$ and $|\fa\cap\accentset{\circ}{\omega}|=|\fa|$.
\end{definition}
While $\SetPair$ covers interelement faces internally, $\tri$ does not.  As a 
consequence, the local best errors associated with $\SetPair$ take into account 
the continuity constraint across interelement faces, a feature that is 
crucial for robustness by the observations in \S\ref{S:counterexample}.

\section{Robust localization to pairs of elements}
\label{S:pair}
%
%
The purpose of this section is to prove the first part of \eqref{replacements}.
The reaction-diffusion norm \eqref{rd-norm} has the $L^2$-norm and the 
$H^1$-seminorm as limiting cases.  We consider these cases first in a unified 
manner by applying Proposition \ref{P:key_estimate} with the same covering; the 
associated interpolation operators differ only in the involved local best 
approximations.  

Throughout this section the mesh $\tri$ is face-connected and, in view of 
\S\ref{S:counterexample}, we choose the covering 
$\SetPair=\{\pair{\fa}\}_{\fa\in\faces}$ and start with the $L^2$-norm, which 
appears the more critical limiting case.

\subsection{Pure reaction norm}
\label{S:L2-norm}
We first introduce an interpolation operator and then show that it is 
locally near best with respect to 
$\tnorm{\cdot}_{\omega}=\norm{\cdot}_{L^2(\omega)}$, $\omega\in\SetPair$.
For simplicity, we write $\mathcal{P}_{\fa}u$ for the best approximation 
to $u$ in $\LFEspace{\pair{\fa}}$ with respect 
to $\norm{\cdot}_{\pair{\fa}}=\norm{\cdot}_{L^2(\pair{\fa})}$.

The definition of the interpolation operator relies on a classification of the 
nodes.  For nodes that are interior to an element, we define the corresponding 
nodal value of $\interp$ with the help of a best approximation.  For the other 
nodes that belong to several elements or are on the boundary $\partial\Omega$, 
we use the averaging technique of L.\ R.\ Scott and S.\ Zhang 
\cite{Scott.Zhang:90}.  More precisely: for every element 
$\el\in\tri$, we fix a face $\fa=\fa_\el\in\faces$ such that 
$\fa_\el\subset\partial\el$ and, moreover, for every 
$z\in\mathcal{N}\cap\Sigma$ with $\Sigma:=\cup_{\el´\in\tri}\partial\el$, we 
fix an element $\elAv\in\tri$ such that $z\in\partial \elAv$.  Given $u\in 
H^1(\Omega)$, we then set
\begin{subequations}
\label{D:interp}
\begin{equation}
 \interp^0 u
 :=
 \sum_{z\in\mathcal{N}}u_z\phi_z 
\end{equation}
where
\begin{equation}
\label{D:nodalval}
 u_z
 =
 \left\{
 \begin{array}{ll}
  \mathcal{P}_{\fa_\el}u(z)
   &\text{if } z\in\mathcal{N}_{\accentset{\circ}{\el}}
    \text{ for some }\el\in\tri,
 \\[2mm]
 \displaystyle \int_{\elAv} u\psi_z^{\elAv}
  &\text{if } z\in\mathcal{N}\cap\Sigma.
 \end{array}
 \right.
\end{equation}
\end{subequations}
Notice that in general $\interp^0 u|_\el$ depends not only on $u|_\el$ but also 
on $u|_{\widetilde\el}$ for neighbouring elements $\widetilde\el$.

In order to verify that $\interp^0$ is locally near best, we fix an element 
$\el\in\tri$, write $\fa:=\fa_\el$ for short and choose $\omega=\pair{\fa}$ 
in \eqref{hp1}. This is an admissible choice for $\omega$ since 
$\pair{\fa}\supseteq\el$. It is also a natural choice because in this way, for 
every $z\in \mathcal{N}_{\accentset{\circ}{\el}}$, we have
\begin{equation}
\label{interior_nodes}
|\mathcal{P}_{\fa}u(z)-\interp^0 u(z)|=0.
\end{equation}
Otherwise, if $z\in\mathcal{N}_{\el}\cap\Sigma$, we exploit Proposition 
\ref{P:path}.  Since $\tri$ is face-connected, there exists a 
finite sequence of faces $\{\fa_j\}_{j=1}^n$ such that the corresponding 
sequence $\{\omega_j\}_{j=1}^{n}:=\{\pair{\fa_j}\}_{j=1}^n\subset\SetPair$ 
satisfies
\begin{itemize}
\item $\omega_1=\pair{\fa}\supseteq\el$ and $\omega_n\supseteq \elAv$,
\item each intersection $\el_j:=\omega_j\cap\omega_{j+1}$ is an element of 
$\tri$ containing $z$.
\end{itemize}
We can therefore apply Proposition \ref{P:path} with $\nu=1$ and get
\begin{equation}
\label{boundary_nodes}
 \left| \mathcal{P}_{\fa}u(z)-\interp^0 u(z) \right|
 \leq
 2\|\Hat{\psi}_{\Hat{z}}\|_{\Hat{\el}}\sum_{j=1}^{n}\frac{|\Hat{\el}|^{1/2}}{
|\el_j|^{1/2}}\norm{\mathcal{P}_{\fa_j}u-u}_{\omega_j}
\end{equation}
with $\el_n:=\el_z$.  Combining \eqref{interior_nodes} and 
\eqref{boundary_nodes} with the scaling property \eqref{phi_scaling}, we obtain
\begin{align*}
 \sum_{z\in\mathcal{N}_\el}\left|
  \mathcal{P}_{\fa}u(z)-\interp^0 u(z)
 \right| &\norm{\phi_z}_{\el}
 \leq
 2\sum_{z\in\mathcal{N}_{\partial\el}}
  \|\Hat{\psi}_{\Hat{z}}\|_{\Hat{\el}}
  \|\Hat{\phi}_{\Hat{z}}\|_{\Hat{\el}}
  \sum_{j=1}^{n} 
   \frac{|\el|^{1/2}}{|\el_j|^{1/2}}\norm{\mathcal{P}_{\fa_j}u-u}_{\omega_j}
\\
 &\leq \constant_0
 \max_{\widetilde{\el}\in\Patch{\el}}
 \frac{|\el|^{1/2}}{|\widetilde{\el}|^{1/2}}
 \sum_{\widetilde{\fa}\in\Skeleton{\el}}
   \norm{\mathcal{P}_{\widetilde{\fa}}u-u}_{\pair{\widetilde{\fa}}},
\end{align*}
where
\begin{equation}
\label{E:K1}
\textstyle
 \constant_0
 =
 \constant_0(\ell,d)
 :=
 2\sum_{\Hat{z}\in\partial\Hat{\el}}
  \|\Hat{\psi}_{\Hat{z}}\|_{\Hat{\el}}
  \|\Hat{\phi}_{\Hat{z}}\|_{\Hat{\el}}
\end{equation}
and $\Patch{\el}$ is the union of the elements $\widetilde{\el}\in\tri$
touching $\el$ and $\Skeleton{\el} := \cup\{\fa\in\faces: 
\fa\cap\partial\el\neq\emptyset\}$ is the skeleton of $\Patch{\el}$.
In order to achieve a bound that is independent of $\el$, we introduce
\begin{equation}
\label{e:mu}
 \mu_\tri
 :=
 \max_{\phantom{\widetilde{I}}\el\in\tri\phantom{\widetilde{I}}}
 \max_{\widetilde{\el}\in\Patch{\el}}
  \frac{|\el|^{1/2}}{|\widetilde{\el}|^{1/2}}
\end{equation}
and arrive at
\begin{align}
\label{key-est-L2}&\qquad
 \sum_{z\in \mathcal{N}_\el}
  |\mathcal{P}_{\fa}u(z)-\interp^0 u(z)|\norm{\phi_z}_{\el}
 \leq
 \mu_\tri \constant_0
 \sum_{\widetilde{\fa}\in\Skeleton{\el}}
  \norm{\mathcal{P}_{\widetilde{\fa}}u-u}_{\pair{\widetilde{\fa}}}.
\end{align}
We observe that $\#\Skeleton{\el}$ is bounded in terms of $d$ and $\Bar{n}$ 
from \eqref{Barn}.  Given $\widetilde\fa\in\faces$, the same holds for 
$\#\{\el\in\tri:\Skeleton{\el}\ni\widetilde\fa\}$.  We can therefore apply 
Proposition \ref{P:key_estimate} and get the following theorem.
\begin{thm}[Localization of best $L^2$-error]
\label{T:LocL2}
For every $u\in L^2(\Omega)$ it holds
\begin{equation*}
 \inf_{v\in \FEspace}\norm{u-v}_\Omega
 \leq 
 C\left(
  \sum_{\fa\in\faces}
  \inf_{P\in\LFEspace{\pair{\fa}}} \norm{u-P}^2_{\pair{\fa}}
 \right)^{\frac12}
 \leq
 C(d+1) \inf_{v\in \FEspace}\norm{u-v}_\Omega
\end{equation*}
where the constant $C$ depends on the polynomial degree $\ell$, the dimension 
$d$, $\Bar{n}$ from \eqref{Barn}, and $\mu_\tri$ from \eqref{E:K1}. 
\end{thm}
Notice that there is no explicit dependence on the shape regularity of $\tri$ 
but a dependence on the local quasi-uniformity of $\tri$ through $\Bar{n}$ and 
$\mu_\tri$. 
\subsection{Pure diffusion seminorm}
\label{S:H1-seminorm}
%
The counterpart of Theorem \ref{T:LocL2} for the $H^1$-seminorm follows from 
\eqref{E:H1dec}.  In this subsection we present an alternative approach relying 
on an interpolation operator that is very close to the one in 
\S\ref{S:L2-norm}.  The obtained inequalities will turn out useful for dealing 
with the reaction-diffusion norm.

Since $\norm{\nabla\cdot}_{\omega}=\norm{\nabla\cdot}_{L^2(\omega)}$ is only a 
seminorm, best approximations in $\FEspace|_\omega$ are only unique up to a 
constant.  This freedom allows to bound the $L^2$-errors appearing in 
Proposition \ref{P:path} by local best errors in the $H^1$-seminorm with the 
help of the Poincar\'e inequality.  We are thus led to the following local 
best approximation operators: given $u\in H^1(\Omega)$ and any $\fa\in\faces$, 
let $\mathcal{R}_\fa$ be the best approximation to $u$ in $S|_{\pair{\fa}}$ 
with respect to $\norm{\nabla\cdot}_{\pair{\fa}}$ such that
\begin{equation}
\label{mean_value}
 \int_{\pair{\fa}}\mathcal{R}_{\fa}u
 =
 \int_{\pair{\fa}}u.
\end{equation}
The interpolation operator $\interp^\infty$ is then given by \eqref{D:interp} 
where  $\mathcal{P}_{\fa_\el}$ is replaced by $\mathcal{R}_{\fa_\el}$.  
Consequently, $\interp^0$ and $\interp^\infty$ differ only in the involved 
local best approximations.  Before embarking on the proof that $\interp^\infty$ 
is locally near best, we provide the following tailor-made Poincar\'e 
inequality.

\begin{lem}[Poincar\'e inequality for element pairs]
\label{R:poin}
Let $\omega$ be the union of two adjacent elements $\el_1$, $\el_2$ sharing a 
face $\fa=\el_1\cap\el_2$. For every $v\in H^1(\omega)$ it holds
\[
 \norm{v-\frac{1}{|\omega|}\int_\omega v }_{\omega}
 \leq
 C_P h_\omega\norm{\nabla v}_\omega,
\]
where $C_P\leq\left(\frac{1}{\pi^2}+\frac{1}{d^2}\right)^{1/2}$ and 
$h_\omega:=\max\{\diam(\el_1),\diam(\el_2)\}$.
\end{lem}
\begin{proof}
Since the mean value on $\omega$ is the best approximating constant with 
respect to the $L^2$-norm, we can substitute it with the mean value on the 
common face $\fa$:
\[
 \norm{v-\frac{1}{|\omega|}\int_\omega v}_\omega^2
 \leq
 \norm{v-\frac{1}{|\fa|}\int_\fa v}_\omega^2
 =
 \sum_{i=1}^2\norm{v-\frac{1}{|\fa|}\int_\fa v}_{\el_i}^2.
\]
Thanks to the trace identity from \cite[Proposition 4.2]{Veeser.Verfuerth:09}, 
we may write
\begin{equation}
\label{e:trace}
 \frac{1}{|\fa|}\int_\fa v 
 =
 \frac{1}{|\el_i|}\int_{\el_i}v
 +
 \frac{1}{d|\el_i|}\int_{ \el_i}\mathbf{q}_{i,\fa}\cdot\nabla v,
\end{equation}
where $\mathbf{q}_{i,\fa}(x):=x-z_i$ and $z_i$ is the vertex of $\el_i$ 
opposite to $\fa$.  Moreover the classical Poincar\'e inequality on convex 
domains, see \cite{Bebendorf:03,Payne.Weinberger:60}, implies
\begin{align*}
 &\norm{
 v - \frac{1}{|\el_i|}\int_{\el_i}v
  - \frac{1}{d|\el_i|}\int_{\el_i}\mathbf{q}_{i,\fa}\cdot\nabla v
 }_{\el_i}^2
\\&\qquad
 =
 \norm{v-\frac{1}{|\el_i|}\int_{\el_i}v}_{\el_i}^2
 +
 \frac{1}{d^2|\el_i|} \left(
  \int_{\el_i}\mathbf{q}_{i,\fa}\cdot\nabla v
 \right)^2
\\ &\qquad
 \leq 
 \diam(\el_i)^2 \left(
  \frac{1}{\pi^2}+\frac{1}{d^2}
 \right)\norm{\nabla v}_{\el_i}^2.
\end{align*}
Combining \eqref{e:trace} and the two inequalities yields the claim.
\end{proof}
In order to show that $\interp^\infty$ is locally near best with respect to 
$\norm{\nabla\cdot}_{\pair{\fa}}$, $\fa\in\faces$, we fix an element 
$\el\in\tri$, write $\fa:=\fa_\el$ for short and, as in \S\ref{S:L2-norm}, we 
choose $\omega=\pair{\fa}$ in \eqref{hp1}. If $z\in 
\mathcal{N}_{\accentset{\circ}{\el}}$, we have again
\begin{equation}
\label{interior_nodesH1}
 |\mathcal{R}_{\fa}u(z)-\interp^\infty u(z)|
 =
 0.
\end{equation}
Otherwise, if $z\in\mathcal{N}_{\el}\cap\Sigma$, we apply also Proposition 
\ref{P:path} with $\nu=1$ and obtain
\[
 |\mathcal{R}_{\fa}u(z)-\interp^\infty u(z)|
 \leq
 2\|\Hat{\psi}_{\Hat{z}}\|_{\Hat{\el}}
 \sum_{j=1}^{n}\frac{|\Hat{\el}|^{1/2}}{|\el_j|^{1/2}}
  \norm{\mathcal{R}_{\fa_j}u-u}_{\omega_j}
\]
with $\el_n=\el_z$.  Here we invoke the Poincar\'e inequality Lemma 
\ref{R:poin}, which yields
\begin{align}
\label{boundary_nodesH1}
|\mathcal{R}_{\fa}u(z)-\interp^\infty u(z)|
 \leq 
 2 C_{P} \|\Hat{\psi}_{\Hat{z}}\|_{\Hat{\el}}
 \sum_{j=1}^{n}
  \hE{\fa_j}\frac{|\Hat{\el}|^{1/2}}{|\el_j|^{1/2}}
  \norm{\nabla(\mathcal{R}_{\fa_j}u-u)}_{\omega_j}.
\end{align}
with $\hE{\fa}:=\max_{\el\in\tri,\el\subseteq\pair{\fa}} \diam(\el)$.  
Combining \eqref{interior_nodesH1} and \eqref{boundary_nodesH1} with 
\eqref{phi_scaling_grad} leads to
\begin{align*}
 &\sum_{z\in \mathcal{N}_{\el}}
 |\mathcal{R}_{\fa}u(z)-\interp^\infty u(z)|\norm{\nabla\phi_z}_{\el}
\\ &\qquad
 \leq 
 2\Hat{h} \sum_{z\in\mathcal{N}_{\partial\el}}
 \|\Hat{\psi}_{\Hat{z}}\|_{\Hat{\el}}
 \|\nabla\Hat{\phi}_{\Hat{z}}\|_{\Hat{\el}}
 C_{P} \sum_{j=1}^n 
  \frac{\hE{\fa_j}|\el|^{1/2}}{\rho_\el|\el_j|^{1/2}}
  \norm{\nabla(\mathcal{R}_{\fa_j}u-u)}_{\pair{\fa_j}}
\\ &\qquad
 \leq 
 \constant_\infty
 \max_{\widetilde{\el}\in\Patch{\el}}
  \frac{|\el|^{1/2}}{|\widetilde{\el}|^{1/2}}
 \max_{\widetilde{\el}\in\Patch{\el}}
  \frac{h_{\widetilde{\el}}}{\rho_\el}
 \sum_{\widetilde{\fa}\in\Skeleton{\el}}
 \norm{\nabla(\mathcal{R}_{\widetilde{\fa}}u-u)}_{\pair{\widetilde{\fa}}},
\end{align*}
where $h_{\widetilde{\el}}:=\diam(\widetilde{\el})$ and
\[
 \textstyle
 \constant_\infty=\constant_\infty(\ell,d)
 :=
 2 \sqrt2 C_P
 \sum_{\Hat{z}\in\partial\Hat{\el}}
  \|\Hat{\psi}_{\Hat{z}}\|_{\Hat{\el}} 
  \|\nabla\Hat{\phi}_{\Hat{z}}\|_{\Hat{\el}}.
\]
We thus see that the use of the Poincar\'e inequality compensates the 
scaling of $\norm{\nabla \phi_z}_{\el}$, if the shape parameter
\begin{equation}
\label{e:sigma}
 \ShapePar_\tri
 :=
 \max_{\phantom{\widetilde{I}}\el\in\tri\phantom{\widetilde{I}}}
 \max_{\widetilde{\el}\in\Patch{\el}}
  \frac{h_{\widetilde{\el}}}{\rho_\el}.
\end{equation}
of $\tri$ is moderate.  We therefore have
\begin{equation}
\label{key-est-H1}
\begin{aligned}
 &\sum_{z\in \mathcal{N}_{\el}}
  |\mathcal{R}_{\fa}u(z)-\interp^\infty u(z)|\norm{\nabla\phi_z}_{\el}
\\ &\qquad
 \leq
 \ShapePar_\tri\mu_\tri \constant_\infty
 \sum_{\widetilde{\fa}\in\Skeleton{\el}}
  \norm{\nabla(\mathcal{R}_{\widetilde{\fa}}u-u)}_{\pair{\widetilde{\fa}}}
\end{aligned}
\end{equation}
and can apply Proposition \ref{P:key_estimate}.  Apart from the dependencies 
listed in Theorem \ref{T:LocL2}, the involved constant depends in addition on 
the shape coefficient of $\tri$ in \eqref{e:sigma}.

\subsection{Reaction-diffusion Norm}
\label{S:eps-norm}
%
%
We now turn to the main result of this section: the robust localization of the 
best error in the reaction-diffusion norm.  To this end, we follow the lines of 
\S\ref{S:L2-norm} and \S\ref{S:H1-seminorm}, combining their results.  For 
simplicity, we write $\tnorm{\cdot}_\omega$ for 
$(\norm{\cdot}^2_{\omega}+\veps\norm{\nabla\cdot}_{\omega}^2)^{1/2}$.

Here we use the following local best approximation operators: given $u\in 
H^1(\Omega)$ and $\fa\in\faces$, let $\RitzP{\fa}u$ be the best approximation 
in $\LFEspace{\pair{\fa}}$ with respect to the norm 
$\tnorm{\cdot}_{\pair{\fa}}$.  Then, for every $v\in \LFEspace{\pair{\fa}}$, 
there holds
\[
 \veps\int_{\pair{\fa}}\nabla u \cdot \nabla v
 +
 \int_{\pair{\fa}}uv
 =
 \veps\int_{\pair{\fa}}\nabla \RitzP{\fa} u \cdot \nabla v
 +
 \int_{\pair{\fa}}\RitzP{\fa} u.
\]
Testing with $v=1$ yields
\begin{equation}
\label{mean_value_RD}
 \int_{\pair{\fa}}\RitzP{\fa}u
 =
 \int_{\pair{\fa}}u,
\end{equation}
which shows that \eqref{mean_value} is a natural choice.  The interpolation 
operator $\interp^\veps$ is then given by \eqref{D:interp} where 
$\mathcal{P}_{\fa_\el}$ is replaced by $\RitzP{\fa_\el}$.  Thus, 
$\interp^\veps$ differs from $\interp^0$ and $\interp^\infty$ only in the 
choice 
of the local best approximations.  This and \eqref{mean_value_RD} entail that 
the counterparts of \eqref{key-est-L2} and \eqref{key-est-H1} for 
$\interp^\veps$ and $\RitzP{\fa_\el}$ hold.

In order to show that $\interp^\veps$ is locally near best in a robust manner,
we again fix $\el\in\tri$, write $\fa:=\fa_\el$ for short and choose 
$\omega=\pair{\fa}$ in \eqref{hp1}.  Due to the equivalence of norms on a 
finite dimensional space, there exist constants $\cnorm=\cnorm(\ell,d)$ and 
$\Cnorm=\Cnorm(\ell,d)$ such that, for every reference node $\Hat{z}$,
\[
 \cnorm\|\Hat{\phi}_{\Hat{z}}\|_{\Hat{\el}}
 \leq
 \|\nabla\Hat{\phi}_{\Hat{z}}\|_{\Hat{\el}}
 \leq
 \Cnorm\|\Hat{\phi}_{\Hat{z}}\|_{\Hat{\el}}.
\]
Using \eqref{phi_scaling} and \eqref{phi_scaling_grad}, we derive from this
\begin{equation}
\label{byL2}
 \tnorm{\phi_z}_\el
 \leq
 \left(
  1 + \veps \frac{\Cnorm^2\Hat{h}^2}{\rho_\el^2}
 \right)^{1/2}
 \norm{\phi_z}_\el
\end{equation}
on one hand and
\begin{equation}
\label{byH1}
 \tnorm{\phi_z}_\el
 \leq
 \left(
  1 + \frac1{\veps} \frac{h_\el^2}{\cnorm^2 \Hat{\rho}^2}
 \right)^{1/2}
 \veps^{1/2} \norm{\nabla \phi_z}_\el
\end{equation}
on the other.  Inequality \eqref{byL2} and the counterpart of 
\eqref{key-est-L2} imply
\begin{align*}
 \sum_{z\in \mathcal{N}_{\el}}
 |\RitzP{\fa}u(z)-\interp^\veps u(z)| \, \tnorm{\phi_z}_{\el}
 \leq \mu_\tri \constant_0
 \sqrt{1+\veps\frac{\Cnorm^2\Hat{h}^2}{\rho_\el^2}}
 \sum_{\widetilde{\fa}\in\Skeleton{\el}}\!\!
  \norm{\RitzP{\widetilde{\fa}}u-u}_{\pair{\widetilde{\fa}}},
\end{align*}
while \eqref{byH1} and the counterpart of \eqref{key-est-H1} give
\begin{multline*}
 \sum_{z\in \mathcal{N}_{\el}}
  |\RitzP{\fa}u(z)-\interp^\veps u(z)| \, \tnorm{\phi_z}
\\
 \leq 
 \mu_\tri \ShapePar_\tri 
 \constant_\infty 
 \sqrt{1+\frac1{\veps} \frac{h_\el^2}{\cnorm^2 \Hat{\rho}^2}} 
 \sum_{\widetilde{\fa}\in\Skeleton{\el}} 
  \veps^{1/2}
  \norm{\nabla(\RitzP{\widetilde{\fa}}u-u)}_{\pair{\widetilde{\fa}}}.
\end{multline*}
Combining the last two inequalities, we arrive at
\begin{equation}
\label{fin_eps}
 \sum_{z\in \mathcal{N}_{\el}}
  |\RitzP{\fa}u(z)-\interp^\veps u(z)| \, \tnorm{\phi_z}
 \leq 
 \mu_\tri \constant_\veps 
 \sum_{\widetilde{\fa}\in\Skeleton{\el}}
 \tnorm{\RitzP{\widetilde{\fa}}u-u}_{\pair{\widetilde{\fa}}}
\end{equation} 
where
\[
 \constant_\veps
 :=
 \min\left\{
  \constant_0
  \sqrt{1+\veps\frac{\Cnorm^2\Hat{h}^2}{\rho_\el^2}},
  \constant_\infty
  \ShapePar_\tri 
  \sqrt{1+\frac1{\veps} \frac{h_\el^2}{\cnorm^2 \Hat{\rho}^2}} 
 \right\}
\]
satisfies $\lim_{\veps\searrow0} \constant_\veps = \constant_0$ and  
$\lim_{\veps\nearrow\infty} \constant_\veps = \constant_\infty$ as well as
\[
 \constant_\veps
 \leq 
 \max\left\{
  \constant_0,
  \constant_\infty \ShapePar_\tri
 \right\}
 \sqrt{1 + \frac{\Cnorm}{\cnorm}
   \frac{\Hat{h}}{\Hat{\rho}} \frac{h_\el}{\rho_\el} }.
\]
Using this in Proposition \ref{P:key_estimate} provides the main result of this 
section:
\begin{thm}[Robust localization for reaction-diffusion norm]
\label{T:pair_dec}
There is a constant $C$ such that for any $\veps>0$ and every $u\in 
H^1(\Omega)$ it holds
\begin{multline*}
 \inf_{v\in \FEspace}\tnorm{u-v}_\Omega
 \leq 
 C\left(
  \sum_{\fa\in\faces}\inf_{P\in\LFEspace{\pair{\fa}}}
 \tnorm{u-P}^2_{\pair{\fa}}
 \right)^{\frac12}
 \leq
 C(d+1) \inf_{v\in \FEspace}\tnorm{u-v}_\Omega
\end{multline*}
where $\tnorm{\cdot}$ stands for $\big(\norm{\cdot}^2 + 
\veps\norm{\nabla\cdot}^2\big)^{1/2}$.   The constant $C$ is bounded in terms 
of the polynomial degree $\ell$, the dimension $d$, $\Bar{n}$ from 
\eqref{Barn}, $\mu_\tri$ from \eqref{E:K1}, and the shape parameter 
$\ShapePar_\tri$ from \eqref{e:sigma}.  If $\veps$ is small with respect to 
$\min_{\el\in\tri} \rho_\el^2$, the dependence on the shape parameter 
$\ShapePar_\tri$ disappears.
\end{thm}

\section{Local best errors: single elements versus pairs}
\label{alternative}
%
%
According to \S\ref{S:counterexample} and \S\ref{S:eps-norm}, the best errors 
in the reaction-dffusion norm \eqref{rd-norm} on single elements do not provide 
a robust localization, while those on pairs of elements do.  This section 
analyzes their difference. As side-products, we derive the second part of 
\eqref{replacements} and an alternative way to compute best errors on element 
pairs.

Throughout this section $\omega$ is the union of two simplices $\el_1$ and 
$\el_2$ with common face $\fa=\el_1\cap\el_2$.  For example, 
$\omega=\pair{\fa}$ with $\fa\in\faces$.  We write $h_i:=h_{\el_i}$ 
and $\rho_i:=\rho_{\el_i}$ for short, $i=1,2$, and set 
\[
 \ShapePar_\fa
 :=
 \frac{\max\{h_1,h_2\}}{\min\{\rho_1,\rho_2\}}
\quad\text{and}\quad
 h_\fa
 :=
 \frac{\min\{|\el_1|,|\el_2|\}}{|\fa|}.
\]
Moreover denote by $P_\fa$ the best approximation in 
$S|_\omega:=S^{\ell,0}(\{\el_1,\el_2\})$ with respect to the reaction-diffusion 
norm $\tnorm{\cdot}_\omega$, whence
\begin{equation}
\label{PE_ba}
 \tnorm{u-P_\fa}_{\omega}
 =
 \inf_{P\in S|_\omega}\tnorm{u-P}_{\omega}. 
\end{equation}
Section \ref{S:counterexample} shows that the best errors on the single 
elements are missing something for a robust upper bound.  We first determine a 
quantity that provides a remedy.
\begin{lem}[Jump augmentation]
\label{L:UbdWithJump}
Let $C_{P}>0$ be a constant and $u\in H^1(\omega)$.  If 
$P_i\in\mathbb{P}_\ell(\el_i)$ satisfy
\begin{equation}
\label{A:Poincare}
 \norm{u - P_i}_{\el_i}
 \leq
 C_{P} h_i \norm{\nabla(u - P_i)}_{\el_i}, 
 \qquad
 i=1,2,
\end{equation}
then there holds
\[
 \tnorm{u-P_\fa}_\omega
 \leq
 C \left(
  h_\fa^{1/2} \norm{P_1-P_2}_\fa
  + \sum_{i=1}^2 \tnorm{u-P_i}_{\el_i}
 \right),
\]
where $C$ depends on $C_{P}$, $\ell$, $d$, $\ShapePar_\fa$, 
but is independent of $\veps$.
\end{lem}
\begin{proof}
We may assume $|\el_2|\geq|\el_1|$ without loss of generality. Define 
$\widetilde{P}\in S|_\omega$ by
\begin{equation*}
 \widetilde{P}(z)
 :=
 \left\{\begin{array}{ll}
  P_i(z)& z\in\mathcal{N}_{\el_i\setminus \fa} \text{ with }i\in\{1,2\},
 \\[2mm]
  \displaystyle P_2(z)&z\in\mathcal{N}_\fa.
 \end{array}\right.
\end{equation*}
Thanks to \eqref{PE_ba} we have 
$\tnorm{u-P_\fa}_{\omega}\leq\tnorm{u-\widetilde{P}}_{\omega}$.  Observing
\begin{equation*}
 \tnorm{u-\widetilde{P}}_{\el_2}
 =
 \tnorm{u-P_2}_{\el_2}
\quad\text{and}\quad
 \tnorm{u-\widetilde{P}}_{\el_1}
 \leq
 \tnorm{u-P_1}_{\el_1}+\tnorm{P_1-\widetilde{P}}_{\el_1},
\end{equation*}
we are left with establishing a suitable bound for 
$\tnorm{P_1-\widetilde{P}}_{\el_1}$.  To this end, we proceed similarly to 
\S\ref{S:eps-norm}.  We expand $P_1-\widetilde{P}$ with respect to the nodal 
basis functions on $\el_1$ and obtain
\begin{align}
\label{expansion}
 \tnorm{P_1-\widetilde{P}}_{\el_1}&
 \leq
 \textstyle
 \sum_{z\in\mathcal{N}_\fa}
 |P_1(z)-P_2(z)| \, \tnorm{\phi_z}_{\el_1}
\end{align}
because $|P_1(z)-\widetilde{P}(z)|=0$ for every 
$z\in\mathcal{N}_{\el_1\setminus\fa}$.  For the shared nodes 
$z\in\mathcal{N}_\fa$, we have
\begin{align}
\label{diffPiPt}
 |P_1(z)-P_2(z)|
 &\leq
 \int_\fa|(P_1-P_2)\psi_z^\fa|
 \leq
 \frac{|\Hat{\fa}|^{1/2}}{|\fa|^{1/2}}
  \norm{P_1-P_2}_\fa \|\Hat{\psi}_{\Hat{z}}\|_{\Hat{\fa}}
\end{align}
by \eqref{psi_scaling}.  In view of \eqref{byL2} and \eqref{phi_scaling}, the 
last two inequalities imply
\begin{equation}
\label{byJump} 
\begin{aligned}
 \tnorm{P_1-\widetilde{P}}_{\el_1}
 &\leq
 \sum_{z\in\mathcal{N}_\fa}
  |P_1(z)-P_2(z)|
  \sqrt{1+\veps\frac{\Cnorm^2 \Hat{h}^2}{\rho^2_1}}
  \norm{\phi_z}_{\el_1}
\\
 &\leq
 m_0 
 \sqrt{1+\veps\frac{\Cnorm^2 \Hat{h}^2}{\rho^2_1}}
 h_\fa^{1/2}
 \norm{P_1-P_2}_\fa
\end{aligned}
\end{equation}
with
\[ 
 m_0
 =
 \frac{|\Hat{\fa}|^{1/2}}{|\Hat{\el}|^{1/2}}
 \sum_{\Hat{z}\in\Hat{\fa}}
  \|\Hat{\psi}_{\Hat{z}}\|_{\Hat{\fa}}
  \|\Hat{\phi}_{\Hat{z}}\|_{\Hat{\el}}.
\]
To derive an alternative bound when $\veps$ is big, we first observe
that the trace identity \eqref{e:trace} and the Poincar\'e inequalities 
\eqref{A:Poincare} yield 
\begin{equation}
\label{e:trace+poin}
 \norm{u-P_i}^2_\fa
 \leq
 C_{P} \left(
  C_{P} +\frac{2}{d}
 \right)
 \frac{ h^2_i |\fa|}{|\el_i|}
  \norm{\nabla(u-P_i)}^2_{\el_i}, \quad i=1,2.
\end{equation}
Using this in \eqref{diffPiPt} gives 
\begin{align*}
 |P_1(z)-P_2(z)|
 &\leq
 \sqrt{
  C_{P} \left(
   C_{P} +\frac{2}{d}
  \right)
 |\Hat{\fa}|
 } \,
 \|\Hat{\psi}_{\Hat{z}}\|_{\Hat{\fa}}
 \sum_{j=1}^2
  \frac{h_{j}}{|\el_j|^{1/2}}\norm{\nabla(u-P_j)}_{\el_j},
\end{align*}
which together with \eqref{expansion}, \eqref{byH1} and 
\eqref{phi_scaling_grad} implies
\begin{equation}
\label{byTrace} 
\begin{aligned}
 \tnorm{P_1-\widetilde{P}}_{\el_1}
 &\leq
 \sum_{z\in\mathcal{N}_\fa}
  |P_1(z)-P_2(z)|
  \sqrt{1 + \frac1\veps \frac{h_1^2}{\cnorm^2 \Hat{\rho}^2}} \, 
  \veps^{1/2} \|\nabla\phi_{z}\|_{\el_1}
\\
 &\leq
 m_\infty
 \sqrt{1 + \frac1\veps \frac{h_1^2}{\cnorm^2 \Hat{\rho}^2}} \, 
 \sum_{j=1}^2
  \frac{h_{j}}{\rho_1}
  \veps^{1/2} \norm{\nabla(u-P_j)}_{\el_j},
\end{aligned}
\end{equation}
where 
\[
 m_\infty
 =
 m_\infty(C_{P},\ell,d)
 :=
 \sqrt{
  C_{P} \left(
   C_{P} +\frac{2}{d}
  \right)
 \frac{|\Hat{\fa}|}{|\Hat{\el}|}
 } \, \Hat{h}
 \sum_{\Hat{z}\in\Hat{E}}
 \|\Hat{\psi}_{\Hat{z}}\|_{\Hat{\fa}}
 \|\nabla\Hat{\phi}_{\Hat{z}}\|_{\Hat{\el}}.
\]
Combining \eqref{byJump} and \eqref{byTrace} yields
\[
 \tnorm{P_1-\widetilde P}_{\el_1}
 \leq 
 C\left(
  h_\fa^{1/2} \norm{P_1-P_2}_\fa
  +
  \sum_{i=1}^2\tnorm{u-P_i}_{\el_i}
 \right)
\]
with
\[
 C
 =
 \max\{ m_0, m_\infty \sigma_\fa \}
 \sqrt{1 + \frac{\Cnorm}{\cnorm} \frac{\Hat{h}}{\Hat{\rho}}
  \frac{h_1}{\rho_1} }
\]
and so the claimed inequality is established.
\end{proof}
Next, we check that the jump term in the upper bound in Lemma 
\ref{L:UbdWithJump} does not overestimate, if we choose 
$P_i\in\mathbb{P}_\ell(\el_i)$ as the best approximations to $u$ with respect 
to 
the reaction-diffusion norm $\tnorm{\cdot}$, that is
\begin{align}
\label{Pi_ba}
 \tnorm{u-P_i}_{\el_i}
 &=
 \inf_{P\in\mathbb{P}_\ell(\el_i)}\tnorm{u-P}_{\el_i}
 \quad i=1,2.
\end{align}
This leads to the following characterization of the best error on a pair by the 
best approximation on its single elements.
\begin{thm}[Sharp jump augmentation]
\label{T:ba-elm-pair}
For any $u\in H^1(\omega)$, it holds
\[
 \tnorm{u-P_\fa}_{\omega}
 \approx
 \left(
  h_\fa\norm{P_1-P_2}^2_{\fa}
  + \sum_{i=1}^2 \tnorm{u-P_i}^2_{\el_i}
 \right)^{\frac{1}{2}}
\]
whenever $P_i$ are given by \eqref{Pi_ba}.  The hidden constants depend on 
$\ell$, $d$, $\ShapePar_\fa$, but are independent of $\veps$.
\end{thm}
\begin{proof}
We start by bounding $\tnorm{u-P_\fa}_{\omega}$ from above.  The choice 
\eqref{Pi_ba} implies $\int_{\el_i} P_i = \int_{\el_i} u$; see
\eqref{mean_value_RD}. Therefore the classical Poincar\'e inequality on convex 
domains, see \cite{Bebendorf:03,Payne.Weinberger:60}, ensures that 
\eqref{A:Poincare} holds with $C_{P}\leq\frac1\pi$ and Lemma 
\ref{L:UbdWithJump} yields the desired bound.

In order to bound $\tnorm{u-P_\fa}_{\omega}$ from below, we first observe  
\begin{equation}
\label{E:interior}
 \sum_{i=1}^2\tnorm{u-P_i}_{\el_i}^2
 \leq
 \sum_{i=1}^2\tnorm{u-P_\fa}_{\el_i}^2
 =
 \tnorm{u-P_\fa}^2_{\omega}
\end{equation}
using \eqref{Pi_ba}.  Therefore the critical term is the jump term.  To bound 
it, we first add and subtract $P_\fa$ and use the triangle inequality:
\begin{equation}
\label{InsertPE}
 \norm{P_1-P_2}_\fa
 \leq
 \norm{P_1-P_\fa}_\fa
  + \norm{P_2-P_\fa}_\fa.
\end{equation}
Since $P_i$ and $P_\fa$ are both polynomials on $\fa$, we write their expansion 
with respect to the nodal basis functions on $\fa$.  Every 
$z\in\mathcal{N}_\fa$ is also a node of $\el_i$ and so we can use 
\eqref{psi_prop} on $\el_i$.  Using also the Cauchy-Schwarz inequality, the 
scaling properties \eqref{phi_scaling} and \eqref{psi_scaling}, we derive the
following explicit inverse inequality:
\begin{align*}
 h_\fa^{1/2}
 \norm{P_i-P_\fa}_\fa
 &\leq 
 h_\fa^{1/2}
 \sum_{z\in\mathcal{N}_\fa}
  \int_{\el_i}|(P_i-P_\fa)\psi_z^{\el_i}|\norm{\phi_z}_\fa
\\
 &\leq
 \frac{|\el_1|^{1/2}}{|\fa|^{1/2}}
 \sum_{z\in\mathcal{N}_\fa}
  \norm{P_i-P_\fa}_{\el_i}
  \|\Hat{\psi}_{\Hat{z}}\|_{\Hat{\el}}
  \|\Hat{\phi}_{\Hat{z}}\|_{\Hat{\fa}}
  \frac{|\Hat{\el}|^{1/2}}{|\Hat{\fa}|^{1/2}}
  \frac{|\fa|^{1/2}}{|\el_i|^{1/2}}
\\
 &\leq
 \widetilde m_0 \norm{P_i-P_\fa}_{\el_i},
\end{align*}
where we have assumed $|\el_2|\geq|\el_1|$ without loss of generality and
\[
 \widetilde m_0
 =
 \widetilde m_0(\ell,d)
 :=
 \frac{|\Hat{\el}|^{1/2}}{|\Hat{\fa}|^{1/2}}
 \sum_{\Hat{z}\in\Hat{E}}
  \|\Hat{\psi}_{\Hat{z}}\|_{\Hat{\el}}
  \|\Hat{\phi}_{\Hat{z}}\|_{\Hat{\fa}}.
\] 
Inserting $u$ and recalling \eqref{Pi_ba}, we get
\begin{align}
\label{E:jump}
 h_\fa^{1/2}\norm{P_i-P_\fa}_\fa
 &\leq
 2 \widetilde m_0 \tnorm{u-P_\fa}_{\el_i}.
\end{align}
Combining \eqref{E:interior}, \eqref{InsertPE} and \eqref{E:jump}, we conclude
\[
 h_\fa \norm{P_1-P_2}^2_\fa
 +
 \sum_{i=1}^2 \tnorm{u-P_i}^2_{\el_i}
 \leq 
 \left(1+8 \widetilde m_0^2\right) \tnorm{u-P_\fa}^2_{\omega},
\]
where the constant depends only on $\ell$ and $d$.
\end{proof}
Theorem \ref{T:ba-elm-pair} shows that best approximations on elements 
augmented with interelement jumps provide also a robust localization.   
As a consequence, the non-robustness in \S\ref{S:counterexample} is caused 
by the absence of these jump terms.  This fact is illustrated by  
the following corollary for the limiting case $\veps=0$.
\begin{cor}[Sharp jump augmentation for $L^2$]
Let $u\in L^2(\omega)$ and denote by $P_\fa$, $P_i$ the best approximations 
to $u$ in $\LFEspace{\omega}$ and $\mathbb{P}_\ell(\el_i)$, 
respectively with respect to the $L^2$-norm. Then
\[
\norm{u-P_\fa}_{\omega}\approx\left(h_\fa\norm{P_1-P_2}^2_{\fa}+\sum_{i=1}
^2\norm{u-P_i}^2_{\el_i}\right)^{\frac{1}{2}},
\]
where the hidden constants depend on $\ell$ and $d$, but are independent of 
$\ShapePar_{\fa}$.
\end{cor}
\begin{proof}
Consider only the $L^2$-part of the reaction-diffusion norm in the proof of 
Theorem \ref{T:ba-elm-pair}.
\end{proof}

The localization associated with the right-hand side in Theorem 
\ref{T:ba-elm-pair} is less costly to compute than the one in Theorem
\ref{T:pair_dec}.  The jumps between the best approximations however require 
communication between elements.  At the price of a slight overestimation, 
the following proposition, which establishes the second part of 
\eqref{replacements}, uses only best errors on elements and in particular 
avoids this communication. 
\begin{prop}[Trace augmentation]
\label{P:only_el}
For any $u\in H^1(\omega)$, it holds
\[
 \tnorm{u-P_\fa}_{\omega}
 \leq
 C \sum_{i=1}^2
  \inf_{P\in\mathbb{P}_\ell(\el_i)}
  \left(
   \tnorm{u-P}_{\el_i}^2
   +
   \frac{|\el_i|}{|\partial\el_i|}\norm{u-P}_{\partial\el_i}^2
 \right)^{1/2}.
\]
where the constant $C$ depends on $\ell$, $d$, $\ShapePar_\fa$, but is 
independent of $\veps$.
\end{prop}
\begin{proof}
Let $P_i$, $i=1,2$, be the best approximations associated with the two infima 
in the claimed bound.  It suffices to verify that they satisfy 
\eqref{A:Poincare}.  In fact, inserting $u$ in the jump term of Lemma 
\ref{L:UbdWithJump} yields the claim.  Fix $i=1,2$ and write $\el:=\el_i$ and 
$v:=u-P_i$ for short.  Since
\[
 \int_\el v + \frac{|\el|}{|\partial\el|} \int_{\partial\el} v
 = 
 0,
\]
we can write
\[
 \norm{v}_\el^2
 =
 \norm{v-\frac{1}{|\el|}\int_\el v}_\el^2
 +
 \frac1{4|\el|}
 \left(
  \frac{|\el|}{|\partial\el|} \int_{\partial\el} v
  -
  \int_\el v
 \right)^2.
\]
Adding the trace identity, cf.\ \eqref{e:trace}, for every face of $\el$ in a 
suitable weighted manner, we obtain a vector field $q_\el$ such that
\[
 \left|
  \frac{|\el|}{|\partial\el|} \int_{\partial\el} v
  -
  \int_\el v
 \right|
 =
 \left|
  \int_\el q_\el\cdot\nabla v
 \right|
 \leq
 \frac{h_\el}d |\el|^{1/2} \norm{\nabla v}_\el.
\]
Consequently the classical Poincar\'e inequality on convex 
domains, see \cite{Bebendorf:03,Payne.Weinberger:60}, shows 
that \eqref{A:Poincare} holds with $C_{P}\leq \left( \frac1\pi+\frac1{4d^2} 
\right)^{1/2}$.
\end{proof}

\section{Robust localization for tree approximation}
\label{S:i4ta}
%
%
Adaptive tree approximation of P.\ Binev and R.\ DeVore \cite{Binev.DeVore:04} 
constructs near best meshes within a hierarchy by means of so-called local 
error functionals.  In this section we propose local error functionals that are 
suitable for the reaction-diffusion norm and, by modifying Theorem 
\ref{T:pair_dec}, we show that they ensure a robust performance of tree 
approximation.

We start by fixing  our setting for tree approximation.  As refinement 
procedure, we adopt bisection of (tagged) simplices as described, e.g.,  
\cite[\S4.1]{Nochetto.Siebert.Veeser:09} or R.\ Stevenson \cite{Stevenson:08}. 
 Recall that the refinement edge of a simplex is the edge that will be halved 
when the simplex is bisected and that the refinement edges of the two children 
simplices are assigned in a unique manner.

Let $\tri_0$ be a conforming mesh of $\Omega$ into simplices such that the 
so-called matching condition, see e.g.\ 
\cite[\S4.2]{Nochetto.Siebert.Veeser:09}, is satisfied.  Denote by 
$\mathbb{T}_c$ the set of all conforming meshes that can be generated from 
$\tri_0$ by successive bisections.  Thanks to the matching condition it 
contains 
in particular all uniform refinements of $\tri_0$.  Moreover we have a graph 
$\mtree$ where the nodes correspond to simplices and each simplex is connected 
with its two children.  This graph is a forest of infinite binary trees, whose 
roots are the elements of $\tri_0$, and it is often called master tree.  Every 
conforming mesh $\tri\in\mathbb{T}_c$ is represented by a subtree $\tree$ of 
the master tree $\mtree$.  The set of leaves and interior nodes of a tree 
$\tree$ are denoted by $\leaves{\tree}$ and $\intNodes{\tree}$, 
respectively.

A local error functional is a map $e$ that associates a positive real 
$e(\el)\geq0$ to any simplex $\el\in\mtree$.  Roughly speaking, tree 
approximation uses $e$ to construct trees $\tree$ that almost minimize the 
global error functional
\begin{equation}
\label{GErrFct}
 \Err(\tree)
 :=
 \sum_{\el\in\mathcal{L}(\tree)}\err{\el}
\end{equation}
within trees of similar cardinality.

Notice that the local best errors in Theorem \ref{T:pair_dec} cannot be 
combined to define an error functional due to the dependence of $\pair{E}$ on 
$\tri$.  To remedy, we mimic the idea of `minimal ring' in P.\ Binev et al.\ 
\cite{Binev.Dahmen.DeVore:04} and introduce the following variant of 
$\pair{E}$: 
 given a face $\fa$ of any 
simplex $\el\in\mtree$, we define
\[
 \Mpair{\fa}
 :=
 \bigcap_{\tri\in\mathbb{T}_c:\fa\text{ is a face of }\tri} \pair{\fa}.
\]
If $\fa$ is an interelement face of some mesh $\tri\in\mathbb{T}_c$, then 
$\Mpair{\fa}$ is the union of two elements $\el'_1$ and $\el'_2$ that belong 
to some virtual refinement of $\tri$ and are such that $\el'_1\cap\el'_2=\fa$. 
See also Figure \ref{F:min_pair}.
\newrgbcolor{zzttqq}{0.8 0.8 0.8}
\begin{figure}
\begin{center}
\begin{pspicture*}(-0.27,-0.25)(2.25,4.25)
\psline(0,0)(0,4)
\psline(2,0)(2,4)
\psline(0,0)(2,0)
\psline(0,4)(2,4)
\psline(0,0)(2,2)
\psline(0,4)(2,2)
\psline[linewidth=2pt](0,2)(2,2)
\psline(0,4)(0.25,4.25)
\psline(0,4)(0,4.25)
\psline(-0.25,4)(0,4)
\psline(-0.25,1.75)(0,2)
\psline(-0.25,2.25)(0,2)
\psline(-0.25,0)(0,0)
\psline(0,-0.25)(0,0)
\psline(1.75,-0.25)(2,0)
\psline(2,-0.25)(2,0)
\psline(2,0)(2.25,0)
\psline(2,2)(2.25,1.75)
\psline(2,2)(2.25,2.25)
\psline(2,4)(2.25,4)
\psline(2,4)(2,4.25)
\rput[b](1,2.1){$0$}
\rput[r](-0.1,3){$0$}
\rput[r](-0.1,1){$0$}
\rput[bl](1,3.1){$-1$}
\rput[tl](1,0.9){$-1$}
\end{pspicture*}
\quad\quad\quad
\begin{pspicture*}(-0.27,-0.25)(2.25,4.25)
\psline(0,0)(0,4)
\psline(2,0)(2,4)
\psline(0,0)(2,0)
\psline(0,4)(2,4)
\psline(0,0)(2,2)
\psline(0,4)(2,2)
\psline[linewidth=2pt](0,2)(2,2)
\psline(0,2)(2,0)
\psline(0,4)(0.25,4.25)
\psline(0,4)(0,4.25)
\psline(-0.25,4)(0,4)
\psline(-0.25,1.75)(0,2)
\psline(-0.25,2.25)(0,2)
\psline(-0.25,0)(0,0)
\psline(0,-0.25)(0,0)
\psline(1.75,-0.25)(2,0)
\psline(2,-0.25)(2,0)
\psline(2,0)(2.25,0)
\psline(2,2)(2.25,1.75)
\psline(2,2)(2.25,2.25)
\psline(2,4)(2.25,4)
\psline(2,4)(2,4.25)
\rput[b](1,2.1){$0$}
\rput[r](-0.1,3){$0$}
\rput[bl](1,3.1){$-1$}
\rput[tr](0.4,1.4){$1$}
\rput[tl](1.6,1.4){$1$}
\end{pspicture*}
\quad\quad\quad
\begin{pspicture*}(-0.27,-0.25)(2.25,4.25)
\psline(0,0)(0,4)
\psline(2,0)(2,4)
\psline(0,0)(2,0)
\psline(0,4)(2,4)
\psline(0,0)(2,2)
\psline(0,4)(2,2)
\psline[linewidth=2pt](0,2)(2,2)
\psline(0,2)(2,0)
\psline(0,2)(2,4)
\psline(0,4)(0.25,4.25)
\psline(0,4)(0,4.25)
\psline(-0.25,4)(0,4)
\psline(-0.25,1.75)(0,2)
\psline(-0.25,2.25)(0,2)
\psline(-0.25,0)(0,0)
\psline(0,-0.25)(0,0)
\psline(1.75,-0.25)(2,0)
\psline(2,-0.25)(2,0)
\psline(2,0)(2.25,0)
\psline(2,2)(2.25,1.75)
\psline(2,2)(2.25,2.25)
\psline(2,4)(2.25,4)
\psline(2,4)(2,4.25)
\rput[b](1,2.1){$0$}
\rput[tr](0.4,1.4){$1$}
\rput[tl](1.6,1.4){$1$}
\rput[br](0.4,2.6){$1$}
\rput[bl](1.6,2.6){$1$}
\end{pspicture*}
\\[5ex]
\begin{pspicture*}(-0.27,-0.25)(2.25,4.25)
\pspolygon[linestyle=none,fillstyle=solid,fillcolor=zzttqq,opacity=0.33](0,2)(1,
3)(2,2)(1,1)
\psline(0,0)(0,4)
\psline(2,0)(2,4)
\psline(0,0)(2,0)
\psline(0,4)(2,4)
\psline(0,0)(2,2)
\psline(0,4)(2,2)
\psline[linewidth=2pt](0,2)(2,2)
\psline[linestyle=dashed](0,2)(2,0)
\psline[linestyle=dashed](0,2)(2,4)
\psline(0,4)(0.25,4.25)
\psline(0,4)(0,4.25)
\psline(-0.25,4)(0,4)
\psline(-0.25,1.75)(0,2)
\psline(-0.25,2.25)(0,2)
\psline(-0.25,0)(0,0)
\psline(0,-0.25)(0,0)
\psline(1.75,-0.25)(2,0)
\psline(2,-0.25)(2,0)
\psline(2,0)(2.25,0)
\psline(2,2)(2.25,1.75)
\psline(2,2)(2.25,2.25)
\psline(2,4)(2.25,4)
\psline(2,4)(2,4.25)
\rput[b](1,2.1){$0$}
\rput[tr](0.4,1.4){$1$}
\rput[tl](1.6,1.4){$1$}
\rput[br](0.4,2.6){$1$}
\rput[bl](1.6,2.6){$1$}
\end{pspicture*}
\quad\quad\quad
\begin{pspicture*}(-0.27,-0.25)(2.25,4.25)
\pspolygon[linestyle=none,fillstyle=solid,fillcolor=zzttqq,opacity=0.33](0,2)(1,
3)(2,2)(1,1)
\psline(0,0)(0,4)
\psline(2,0)(2,4)
\psline(0,0)(2,0)
\psline(0,4)(2,4)
\psline(0,0)(2,2)
\psline(0,4)(2,2)
\psline[linewidth=2pt](0,2)(2,2)
\psline(0,2)(2,0)
\psline[linestyle=dashed](0,2)(2,4)
\psline(0,4)(0.25,4.25)
\psline(0,4)(0,4.25)
\psline(-0.25,4)(0,4)
\psline(-0.25,1.75)(0,2)
\psline(-0.25,2.25)(0,2)
\psline(-0.25,0)(0,0)
\psline(0,-0.25)(0,0)
\psline(1.75,-0.25)(2,0)
\psline(2,-0.25)(2,0)
\psline(2,0)(2.25,0)
\psline(2,2)(2.25,1.75)
\psline(2,2)(2.25,2.25)
\psline(2,4)(2.25,4)
\psline(2,4)(2,4.25)
\rput[b](1,2.1){$0$}
\rput[tr](0.4,1.4){$1$}
\rput[tl](1.6,1.4){$1$}
\rput[br](0.4,2.6){$1$}
\rput[bl](1.6,2.6){$1$}
\end{pspicture*}
\quad\quad\quad
\begin{pspicture*}(-0.27,-0.25)(2.25,4.25)
\pspolygon[linestyle=none,fillstyle=solid,fillcolor=zzttqq,opacity=0.33](0,2)(1,
3)(2,2)(1,1)
\psline(0,0)(0,4)
\psline(2,0)(2,4)
\psline(0,0)(2,0)
\psline(0,4)(2,4)
\psline(0,0)(2,2)
\psline(0,4)(2,2)
\psline[linewidth=2pt](0,2)(2,2)
\psline(0,2)(2,0)
\psline(0,2)(2,4)
\psline(0,4)(0.25,4.25)
\psline(0,4)(0,4.25)
\psline(-0.25,4)(0,4)
\psline(-0.25,1.75)(0,2)
\psline(-0.25,2.25)(0,2)
\psline(-0.25,0)(0,0)
\psline(0,-0.25)(0,0)
\psline(1.75,-0.25)(2,0)
\psline(2,-0.25)(2,0)
\psline(2,0)(2.25,0)
\psline(2,2)(2.25,1.75)
\psline(2,2)(2.25,2.25)
\psline(2,4)(2.25,4)
\psline(2,4)(2,4.25)
\rput[b](1,2.1){$0$}
\rput[tr](0.4,1.4){$1$}
\rput[tl](1.6,1.4){$1$}
\rput[br](0.4,2.6){$1$}
\rput[bl](1.6,2.6){$1$}
\end{pspicture*}
\caption{\label{F:min_pair} Different triangulations (top row) and minimal pair 
(bottom row, grey shading) associated with an edge $E$ (bold line). The 
labeling is rearranged so that $E$ has label $0$. Dashed lines indicate virtual 
refinement.}
\end{center}
\end{figure}
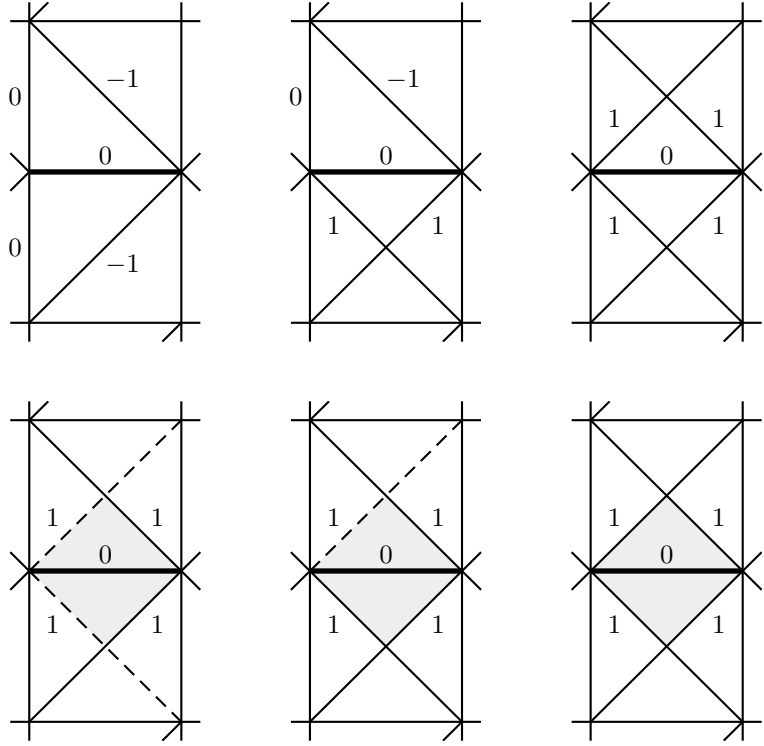
\begin{rem}[Properties of $\omega_\star$]
\label{R:Mpair}
Let $\tri\in\mathbb{T}_c$ be a conforming mesh and denote by $\Afaces$ the set 
of faces, and by $\faces$ its subset of the interelement faces.  There hold:
\begin{enumerate}
 \item[(i)] Let $\el\in\tri$ and $\fa\in\Afaces$ such that 
$\fa\subseteq\partial\el$. Then $\el\subseteq\Mpair{\fa}$ if and only if the 
refinement edge $\ed$ of $\el$ is contained in the face $\fa$.
 \item[(ii)] Let $\fa\in\Afaces$. For every $\el\in\tri$ with 
$\el\subseteq\pair{\fa}$, we have that $\Mpair{\fa}\cap\el$ is an element of 
$\tri$ or of some virtual refinement of $\tri$ so that $|\Mpair{\fa}\cap 
\el|\geq|\el|/2$.
\item[(iii)] For any $\el\in\tri$, there exists $\fa\in\Afaces$ such that 
$\Mpair{\fa}\supseteq\el$. Moreover it is possible to take $\fa\in\faces$, if 
there exists a face $\fa\subset\partial\el$ such that 
$\fa\nsubseteq\partial\Omega$ and $\fa$ contains the refinement edge of $\el$.
 \item[(iv)] The collection $\{\Mpair{\fa}\}_{\fa\in\Afaces}$ is a $d$-finite 
covering of $\tri$.
\end{enumerate}
\end{rem}
\begin{proof}
We start with (i). Since $\fa\subseteq\partial\el$, we have 
$\el\subseteq\pair{\fa}$. We consider the two cases $\ed\subseteq\fa$ and 
$\ed\not\subseteq\fa$ separately.  In the first case, if $\el$ is bisected, 
$\fa$ is also bisected, and there does not exist any conforming refinement 
$\tri'\in\mathbb{T}_c$ of $\tri$ and any descendant $\el'$ of $\el$ such that 
$\el'\subset\omega_{\tri'}(\fa)$.  Hence $\el\subseteq\Mpair{\fa}$. On the 
other hand, consider the second case $\ed\nsubseteq\fa$.  If $\el$ is 
bisected, one of its children $\el'$ still contains $\fa$ and its refinement 
edge is contained in $\fa$. Therefore $\el'\subseteq\Mpair{\fa}$ and 
$\el'\subsetneq\el$ entails $\el\nsubseteq\Mpair{\fa}$.

In order to verify (ii), we note from the proof of (i) that either 
$\Mpair{\fa}\cap\el=\el$ or $\Mpair{\fa}\cap\el=\el'$, where $\el'$ is a child 
of $\el$.  Since bisection yields $|\el'|=|\el|/2$, we have also 
$|\Mpair{\fa} \cap \el|\geq|\el|/2$.

Finally, taking a face which contains the refinement edge of $\el$ and applying 
(i) shows (iii), which then together with (i) implies (iv).
\end{proof}
Motivated by the form \eqref{GErrFct} of the global error functional and the 
fact that the covering in Remark \ref{R:Mpair} (iv) covers faces internally 
(see Definition \ref{e:ric_E}), we introduce the following local error 
functional for the reaction-diffusion norm:
\begin{equation}
\label{def_e}
 \err{\el}
 :=
 \sum_{\fa\text{ face of }\el}\inf_{P\in\LFEspace{\Mpair{\fa}}}
  \tnorm{u-P}^2_{\Mpair{\fa}},
 \qquad
 \el\in\mtree.
\end{equation}
This indeed depends only on $\el$ because each $\Mpair{\fa}$ depends only on 
$\fa$ and the refinement edges in the mesh $\tri_0$.

The finiteness of the covering in Remark \ref{R:Mpair} (iv) is related also the 
following useful property of the local error functional, which 
\cite{Binev.DeVore:04} calls `modified subadditivity'.
\begin{prop}[Weak subadditivity]
\label{P:subadditivity}
The functional $e$ in \eqref{def_e} has the following 
property: if
$\tree\subset\mtree$ is a tree with a single root $\el$ and such that 
$\leaves{\tree}$ is a conforming mesh, then 
\begin{equation}
 \label{subadd_prop}
 \sum_{\el'\in\leaves{\tree}}
  \err{\el'}
 \leq 2d \, \err{\el}.
\end{equation}
\end{prop}
\begin{proof}
For every $\fa$ face of $\el$, denote by $v_\fa$ the 
best approximation in $\LFEspace{\Mpair{\fa}}$ to $u$ with respect to the 
$\tnorm{\cdot}_{\Mpair{\fa}}$-norm. For every 
$\fa'\subset\el'\in\mathcal{L}(\tree)$, there exists $\fa\subset\el$ such that 
$\Mpair{\fa'}\subset\Mpair{\fa}$ and so
\[
 \tnorm{u-v_{\fa'}}_{\Mpair{\fa'}}
 \leq
 \tnorm{u-v_{\fa}}_{\Mpair{\fa'}}.
\]
Since a given point in $\Mpair{\fa}$ is contained in at most $d$ pairs 
$\Mpair{\fa'}$, we thus obtain
\begin{align*}
 \sum_{\el'\in\mathcal{L}(\tree)}e(\el')
 &\leq
 \sum_{\el'\in\mathcal{L}(\tree)}
  \sum_{\phantom{\mathcal{L}}\fa'\text{ face of }\el'}
   \tnorm{u-v_{\fa'}}^2_{\Mpair{\fa'}}
\\
 &\leq
 2 \sum_{\fa'\text{ face of }\mathcal{L}(\tree)}
  \tnorm{u-v_{\fa}}^2_{\Mpair{\fa'}}
 \leq
 2d \sum_{\fa\text{ face of }\el}
  \tnorm{u-v_{\fa}}_{\Mpair{\fa}}^2
\\&=2d\,e(\el). \qedhere
\end{align*}
\end{proof}
Modifying the proof of Theorem \ref{T:pair_dec}, we can relate its 
corresponding global error functional with the best error in $\FEspace$ with 
respect to the reaction-diffusion norm \eqref{rd-norm}.
\begin{thm}[Localization for tree approximation]
\label{T:Mpair}
Let $\Err$ be given by \eqref{GErrFct} with \eqref{def_e}.  For any 
$u\in H^1(\Omega)$ and any conforming mesh $\tri\in\mathbb{T}_c$, it holds
\begin{equation*}
 \inf_{v\in \FEspace^{\ell,0}(\tri)}\tnorm{u-v}_\Omega
 \approx
 \Err(\tree)^{1/2},
\end{equation*}
where $\tree\subset\mtree$ is the tree corresponding to $\tri$ and the hidden 
constants depend only on the polynomial degree $\ell$, the dimension $d$, the 
mesh $\tri_0$, but not on $\veps$.
\end{thm}
\begin{proof}
Writing $\FEspace := \FEspace^{\ell,0}(\tri)$ for short, the results follows 
from the following inequalities:
\begin{equation}
\label{E:Mpair_dec}
 \inf_{v\in \FEspace} \tnorm{u-v}_\Omega
 \leq
 C \left(
  \sum_{\fa\in\Afaces}
   \inf_{P\in\FEspace|_{\Mpair{\fa}}} \tnorm{u-P}_{\Mpair{\fa}}
 \right)^{1/2}
 \leq
 Cd  \inf_{v\in \FEspace} \tnorm{u-v}_\Omega,
\end{equation}
where $C$ depends on $\ell$, $d$, and the shape parameter $\ShapePar_\tri$.  In 
fact, $\Err(\tree)$ regroups only the terms of the sum inside the square root 
of \eqref{E:Mpair_dec} and the shape parameter $\ShapePar_\tri$ is bounded in 
terms of the $\ShapePar_{\tri_0}$; see, e.g.,
\cite[Corollary 4.1]{Nochetto.Siebert.Veeser:09}.

The proof of \eqref{E:Mpair_dec} resembles the one of Theorem \ref{T:pair_dec} 
and we restrict ourselves to emphasize the differences.

In order to define the interpolation operator, we fix $\el_z$ for every 
$z\in\mathcal{N}\cap\Sigma$ as before but, for every $\el\in\tri$, we fix 
$\fa_\el$ such that $\Mpair{\fa_\el}\supset\el$.  The latter allows to 
choose $\omega=\Mpair{\fa_\el}$ in the verification of \eqref{hp1} but requires 
to incorporate the faces on the domain boundary $\partial\Omega$ in the 
localization \eqref{E:Mpair_dec}.  The interpolation operator 
$\interp^\veps_\star$ is then given by \eqref{D:interp} where 
$\mathcal{P}_{\fa_\el}$ is replaced by $\mathcal{R}^\veps_{\star,\fa_\el}$, the 
best approximation operator associated with $S|_{\Mpair{\fa_\el}}$ and the 
reaction-diffusion norm.

Next, we show that $\interp^\veps_\star$ is locally near best with respect to 
the covering $\mathcal{W^*} := \{\Mpair{\fa}\}_{\fa\in\Afaces}$ in Remark 
\ref{R:Mpair} (iv).  To this end, we fix $\el\in\tri$, write $\fa:=\fa_\el$ and 
choose $\omega=\Mpair{\fa}$ in \eqref{hp1}.  Again, we have 
$|\mathcal{R}^\veps_{\star,\fa} u(z) - \interp^\veps_\star u(z)|=0$ for 
$z\in\mathcal{N}_{\accentset{\circ}{\el}}$ and exploit Proposition \ref{P:path} 
for $z\in\mathcal{N}_\el\cap\Sigma$.  For the covering $\mathcal{W}_\star$ 
however, 
the construction of the path of subdomains is more involved.

Since $\tri$ is face-connected, there exists a sequence 
$\{\el_i\}_{i=1}^{r}$ of elements of $\tri$ such that $\el_1=\el$, 
$\el_r=\elAv$, and each intersection $\el_i\cap\el_{i+1}\in\faces$ is an 
interelement face containing $z$.  We write 
$\inters{\fa}_i:=\el_i\cap\el_{i+1}$ 
for the intersections and, for every element $\el_i$, we choose a face 
$\fa_{\el_i}$ such that $\Mpair{\fa_{\el_i}}\supseteq\el_i$.  We then construct 
the path $\{\omega_j\}_{j=1}^n:=\{\Mpair{\fa_j}\}_{j=1}^n$ of subdomains by 
means of the following algorithm:
\begin{itemize}
 \item[]
 \item[] $\fa_1:=\fa$, $j:=1$
 \item[] \texttt{if} $\fa_1\neq\inters{\fa}_1$ \texttt{then}
 \begin{itemize}
  \item[] $\fa_2:=\inters{\fa}_1$, $j:=j+1$
 \end{itemize}
 \item[] \texttt{endif}
 \item[] \texttt{for} $i=1,\dots,r-2$ \texttt{do}
  \begin{itemize}
   \item[] \texttt{if} 
$|\Mpair{\inters{\fa}_i}\cap\Mpair{\inters{\fa}_{i+1}}|\geq |\el_{i+1}|/2$
\texttt{then}
   \begin{itemize}
    \item[] $\fa_{j+1}:=\inters{\fa}_{i+1}$, $j:=j+1$
   \end{itemize}
   \item[] \texttt{else}
   \begin{itemize}
    \item[] $\fa_{j+1}:=\fa_{\el_{i+1}}$, $\fa_{j+2}:=\inters{\fa}_{i+1}$,
$j:=j+2$
   \end{itemize}
   \item[] \texttt{endif}
  \end{itemize}
  \item[] \texttt{endfor}
  \item[] $\fa_{j+1}:=\fa_{\el_r}$
  \item[]
\end{itemize}
In view of Remark \ref{R:Mpair}, this path is admissible for Proposition 
\ref{P:path} with $\nu=1/2$.  We therefore can follow the lines in the proof of 
Theorem \ref{T:pair_dec}, replacing pairs by minimal pairs.  Taking into 
account $\nu=1/2$, $h_{\Mpair{\fa'}}\leq h_{\pair{\fa'}}$ and the fact that 
more faces are involved, we derive

\begin{equation}
\label{fin_epsMp}
\sum_{z\in \mathcal{N}_{\el}}
  |\mathcal{R}^\veps_{\star,\fa} u(z) - \interp^\veps_\star u(z)|
  \tnorm{\phi_z}_{\el}
 \leq
 \sqrt{2} \mu_\tri M_\veps 
 \sum_{\widetilde{\fa}\in\skeleton{\el}} 
  \tnorm{\mathcal{R}^\veps_{\star,\widetilde{\fa}}u-u}_{\Mpair{\widetilde{\fa}
} } .
\end{equation}
with $\skeleton{\el}:=\Skeleton{\el}\cup\{\fa\in\Afaces: 
\fa\subset\partial\Patch{\el}\}$.  Since $\#\skeleton{\el}$ and 
$\#\{\el\in\tri: \fa\in\skeleton{\el}\}$ are still bounded in terms of 
$\Bar{n}$ and $d$, Proposition \ref{P:key_estimate} ensures \eqref{E:Mpair_dec}.
\end{proof}
Let us illustrate the usefulness of the error functional proposed in 
\eqref{def_e}.  To this end, consider the Modified Second Algorithm in 
\cite[\S7]{Binev.DeVore:04} with \eqref{def_e}; one could simplify the 
so-called 'new subdivision rule' therein in our context.  A combination of 
Proposition~\ref{P:subadditivity}, \cite[Theorem 7.2]{Binev.DeVore:04} and 
Theorem~\ref{T:Mpair} yields that any output mesh $\tri$ is near best in the 
following sense:  
\[
 \inf_{v\in\FEspace^{\ell,0}(\tri)} \tnorm{u-v}_\Omega
 \leq
 C E_{c\#\tri},
\]
where
\[
 E_n
 :=
 \min\left\{
  \inf_{v\in\FEspace^{\ell,0}(\tri')} \tnorm{u-v}_\Omega
  : \tri'\in\mathbb{T}_c, \#\tri'\leq n
 \right\}
\]
and $C\geq1$, $c\in(0,1]$ are constants depending on $d$, $\ell$, $\tri_0$ 
but not on $\veps$.  Moreover, counting an evaluation of $e$ with 1 operation, 
the algorithm uses less than $O(\#\tri+\#\tri_0)$ operations to create $\tri$.

\section{Robust localization with Dirichlet boundary conditions}
\label{S:H10}
%
%
In this section we briefly discuss the modifications of our results if the 
boundary values of the target function are imposed on the approximants.  This 
is 
of interest, for example, when conforming finite element methods are applied to 
the homogeneous Dirichlet problem of the reaction-diffusion equation.
For simplicity, we consider target functions in $H^1_0(\Omega)$ approximated by 
elements from $\FEspace_0 := S^{\ell,0}_0(\tri) := S^{\ell,0}(\tri)\cap 
H^1_0(\Omega)$.

Considering $u_\veps=\min\{1,\veps^{-1/2}\dist(\partial\Omega)\}$ as in 
\S\ref{S:counterexample} reveals the following: if any local best error on a 
subdomain $\omega$ with positive $(d-1)$-dimensional Hausdorff measure 
$|\partial\omega\cap\partial\Omega|>0$ does not incorporate  the boundary 
condition, robustness cannot hold.  This suggests the following modification of 
the setting for, e.g., Theorem \ref{T:pair_dec}.  We 
associate to every $\fa\in\faces$ the local space
\begin{equation}
\label{BC:LocSpaces}
 \FEspace_\fa
 :=
 \begin{cases}
  \FEspace_0|_{\pair{\fa}} &\text{if }\fa\in\Bfaces,
 \\
  \FEspace|_{\pair{\fa}} &\text{otherwise,} 
 \end{cases}
\end{equation}
where $\Bfaces:=\{\fa'\in\Afaces: \pair{\fa'} \text{ has a face on 
}\partial\Omega\}$.  Notice that we have $\FEspace_\fa \neq 
\FEspace_0|_{\pair{\fa}}$ if and only if $\fa\not\in\Bfaces$ but 
$\pair{\fa}\cap\partial\Omega\neq\emptyset$.  This however, at least, 
does not create a problem for the second inequality in Theorem \ref{T:pair_dec} 
since $\FEspace_\fa \supset \FEspace_0|_{\pair{\fa}}$.

The interpolation operator $\interp^\veps_0$ now has to vanish on the domain 
boundary $\partial\Omega$.   To this end, the elegant approach of averaging on 
boundary faces in \cite{Scott.Zhang:90}, which has been adopted in 
\cite{Veeser:13}, cannot be applied because the use of traces does not allow 
for \eqref{boundary_nodes} and so for robustness.  We therefore use the 
original approach of suppressing the boundary nodes in P.\ Cl\'ement 
\cite{Clement:75} and set
\[ \textstyle
 \interp^\veps_0 u
 :=
 \sum_{z\in\mathcal{N}_\Omega} u_z \phi_z
\]
with $u_z$ as is \eqref{D:nodalval}, where $\mathcal{P}_{\fa_\el} u$ is 
replaced by the best approximation $\mathcal{R}_{0,\fa_\el}^\veps u$ in 
$\FEspace_{\pair{\fa_\el}}$ to $u$ with respect to the reaction-diffusion norm.
Consequently, $\interp^\veps$ and $\interp^\veps_0$ differ only at boundary 
nodes and at nodes invoking a face in $\Bfaces$.

Nevertheless the counterpart of \eqref{key-est-L2} holds.  To see this, 
consider $\el\in\tri$ such that $\el\cap\partial\Omega$ is non-empty and 
notice that only the boundary nodes $z\in\mathcal{N}_\el\cap\partial\Omega$ are 
critical.  If $\el$ has a face on the domain boundary $\partial\Omega$, then 
the 
same holds for $\pair{\fa_\el}$ and so we have 
$u_z=0=\mathcal{R}_{0,\fa_\el}^\veps u(z)$ for all boundary nodes
$z\in\mathcal{N}_\el\cap\partial\Omega$.  
If the intersection $\el\cap\partial\Omega$ is only a $k$-face with $k< d-1$, 
we 
can find a path $\{\pair{\fa_j}\}_{j=1}^n$ of pairs such that 
$\pair{\fa_1}\supset\el$ and $\pair{\fa_n}$  has a face on $\partial\Omega$.  
Hence we can bound $|\mathcal{R}_{0,\fa_\el}^\veps 
u(z)|=|\mathcal{R}_{0,\fa_\el}^\veps u(z)-\mathcal{R}_{0,\fa_n}^\veps u(z)|$ as 
in the proof of Proposition \ref{P:path} for all boundary nodes 
$z\in\mathcal{N}_\el\cap\partial\Omega$. 

Inequality \eqref{key-est-H1} hinges on a Poincar\'e-type inequality on pairs.  
If $\fa'\in\Bfaces$, then \eqref{mean_value_RD} may not be correct and the 
counterpart of \eqref{key-est-H1} is built on the following Friedrichs' 
inequality, which is a tailor-made variant of Lemma 5.1 in 
\cite{Veeser.Verfuerth:09}.
\begin{lem}[Friedrichs inequality on element pairs]
Let $\omega$ be the union of two adjacent elements $\el_1$, $\el_2$ sharing a 
face $\fa=\el_1\cap\el_2$.  Moreover let $\fa_0$ be a face of $\el_1$. For 
every 
$v\in H^1(\omega)$ with $\int_{\fa_0}v=0$, it holds
\[
 \norm{v}_{\omega}\
 \leq
 C_{F,\omega} h_\omega \norm{\nabla v}_{\omega},
\]
where $h_\omega:=\max\{\diam\el_1,\diam\el_2\}$ and $C_F$ depends on $d$ and 
the shape parameter of $\{\el_1,\el_2\}$.
\end{lem}
\begin{proof}
Adding and subtracting the mean value over the common face $\fa$, we obtain
\[
\norm{v}_{\omega}\leq\norm{v-\frac{1}{|\fa|}\int_\fa 
v}_\omega+|\omega|^{1/2}\left|\frac{1}{|\fa|}\int_{\fa}v\right|
\]
We treat the first term on the right-hand side as in the proof of Lemma 
\ref{R:poin}.  For the second term, we use twice the trace identity 
\eqref{e:trace} to get
\[
 \frac{1}{|\fa|}\int_\fa v
 =
 \frac{1}{d|\el_1|} \int_{\el_1}(z_{\fa_0}-z_{\fa})\cdot\nabla v,
\]
where $z_{\fa_0}$ and $z_{\fa}$ are the vertices opposite to $\fa_0$ and $\fa$, 
respectively.  Hence
\begin{align*}
 |\omega|^{1/2} \left|\frac{1}{|\fa|}\int_{\fa}v\right|
 &\leq 
 \left( |\el_1|^{1/2}+|\el_2|^{1/2} \right)
  \frac{|z_{\fa_0}-z_{\fa}|
  \norm{\nabla v}_{\el_1}}{d|\el_1|^{1/2}}
\\
 &\leq 
 \left(
  \frac{\diam(\el_1)}{d} + 
  \frac{|\el_2|^{1/2}|z_{\fa_0}-z_{\fa}|}{d|\el_1|^{1/2}}
 \right) \norm{\nabla v}_{\el_1}.
\end{align*}
Since
\[
 \frac{|\el_2|^{1/2}|z_{\fa_0}-z_{\fa}|}{|\el_1|^{1/2}}
 \leq 
 \frac{|\fa|^{1/2}\diam(\el_2)^{1/2}|z_{\fa_0}-z_{\fa}|}%
  {|\fa|^{1/2}\dist(z_\fa,\fa)^{1/2}}
 \leq 
 \frac{|z_{\fa_0}-z_{\fa}|^{1/2}}{\dist(z_\fa,\fa)^{1/2}} h_\omega
\]
we conclude
\[
\norm{v}_\omega\leq\left(C_P+\frac{1}{d}+\frac{|z_{\fa_0}-z_{\fa}|^{1/2}}{d\,
\dist(z_\fa,\fa)^{1/2}}\right)h_\omega\norm{\nabla v}_\omega. \qedhere
\]
\end{proof}
We thus obtain the following variant of Theorem \ref{T:pair_dec}.
\begin{thm}[Robust localization with boundary condition]
Using \eqref{BC:LocSpaces}, there holds 
\begin{equation*}
 \inf_{v\in S_0(\tri)} \tnorm{u-v}_\Omega
 \approx 
 \left(
  \sum_{\fa\in\faces}
   \inf_{P\in\FEspace_{\pair{\fa}}} \tnorm{u-P}^2_{\pair{\fa}}
 \right)^{\frac12}
\end{equation*}
for every $u\in H^1_0(\Omega)$.  The hidden constants depend only on the 
polynomial degree $\ell$, the dimension $d$, the shape parameter 
$\ShapePar_\tri$, but not on $\veps$.
\end{thm}
There holds a similar theorem where the covering 
$\{\pair{\fa}\}_{\fa\in\faces}$ of `normal' pairs is replaced by the one 
$\{\Mpair{\fa}\}_{\fa\in\Afaces}$ of minimal pairs.

\end{document}